\documentclass[12pt]{article}
\usepackage{tikz}
\usetikzlibrary{arrows}
\usepackage[framemethod=tikz]{mdframed}
\usepackage{wrapfig}
\usepackage{amsmath,amssymb}
\usepackage{type1cm}
\usepackage{amsthm}
\usepackage{mathrsfs}
\usepackage{mathtools}
\usepackage{enumerate}
\usepackage[all]{xy} 
\usepackage[vcentermath,enableskew]{youngtab}
\usepackage{ytableau}
\usepackage{diagbox}
\AtBeginDocument{%
   \def\MR#1{}
}
\let\originalthebibliography\thebibliography
\renewcommand{\thebibliography}[1]{%
  \originalthebibliography{#1}%
  \small
  \setlength{\itemsep}{2pt}%
  \setlength{\parsep}{0pt}%
  \setlength{\parskip}{0pt}%
}

\DeclareMathOperator{\GL}{GL}

\DeclareMathOperator{\GU}{GU}
\DeclareMathOperator{\GSp}{GSp}

\DeclareMathOperator{\ad}{ad}

\DeclareMathOperator{\Irr}{Irr}

\DeclareMathOperator{\de}{def}

\DeclareMathOperator{\sm}{sm}

\DeclareMathOperator{\pa}{par}
\DeclareMathOperator{\gen}{gen}
\DeclareMathOperator{\depth}{depth}

\newcommand\F{\mathbb{F}}

\newcommand\Fq{\mathbb F_q}
\newcommand\aFq{\overline{\mathbb F}_q}
\newcommand\cA{\mathcal A}

\newcommand\cO{\mathcal O}

\newcommand\cG{\mathcal G}
\newcommand\cF{\mathcal F}

\newcommand\cS{\mathcal S}

\newcommand\Gm{\mathbb G_m}

\newcommand\N{\mathbb N}
\newcommand\Q{\mathbb Q}
\newcommand\R{\mathbb R}
\newcommand\A{\mathbb A}

\newcommand\J{\mathbb J}

\newcommand\Z{\mathbb Z}
\newcommand\tW{\tilde W}

\newcommand\tnu{\tilde \nu}

\newcommand\bS{\mathbb S}
\newcommand\tS{\tilde{\mathbb S}}
\newcommand\inv{\mathrm{inv}}

\newcommand\tp{\mathrm{top}}
\newcommand\ld{\lambda}

\newcommand\ep{\epsilon}

\newcommand\vp{\varpi}
\newcommand\Y{X_*(T)}

\newcommand\la{\langle}
\newcommand\ra{\rangle}

\newcommand\Xl{X^\ld_\mu(b)}
\newcommand\nt{\natural}
\newcommand\cAm{\cA_{\mu,b}}

\newcommand\tP{\tilde \Phi}
\newcommand\ta{\tilde \alpha}
\newcommand\tb{\tilde \beta}
\newcommand\Wa{W_{\mathcal O_{\tilde \alpha}}}
\newcommand\wa{w_{\alpha}}
\newcommand\Oa{\mathcal O_{\alpha}}
\newcommand\w{\mathbf{w}}

\newcommand\el{\ep_\ld}
\newcommand\tle{\trianglelefteq}

\renewcommand\a{\mathbf{a}}
\renewcommand\ln{\lambda^\natural}

\theoremstyle{definition}
\newtheorem{theo}{Theorem}[section]
\newtheorem{prop}[theo]{Proposition}

\newtheorem{lemm}[theo]{Lemma}
\newtheorem{coro}[theo]{Corollary}

\newtheorem{rema}[theo]{Remark}

\newtheorem{thm}{Theorem}[section]

\begin{document}
\title{On $\J$-strata with Parahoric Stabilizers in Affine Deligne-Lusztig Varieties}
\author{Ryosuke Shimada}
\date{}
\maketitle

\begin{abstract}
In this paper, we study the $\J$-stratification of basic affine Deligne-Lusztig varieties for a minuscule cocharacter $\mu$.
This stratification was introduced by Chen-Viehmann \cite{CV18} and has been expected to serve as an interesting tool for studying basic loci in Shimura varieties.
We parametrize the $\J$-strata whose stabilizers in the Frobenius-twisted centralizer group are parahoric by constructing a natural bijection to combinatorial invariants called small cocharacters.
We further prove that the cardinality of these sets is equal to that of a certain subset of the Weyl group orbit of $\mu$.
A relationship with the weakly fully Hodge-Newton decomposability of Chen-Tong \cite{CT22} is also discussed.
\end{abstract}

\section{Introduction}
\label{introduction}

One way to understand arithmetic properties of Shimura varieties is to study the geometry and cohomology of the special fiber of a suitable integral model.
The moduli description of Shimura varieties of Hodge type allows us to define the Newton stratification of the special fiber.
There is a unique closed Newton stratum, which is considered to be particularly important.
This is the so-called basic locus.

Affine Deligne-Lusztig varieties were introduced by Rapoport in \cite{Rapoport05} and play an important role in understanding the geometric and arithmetic properties of Shimura varieties.
The Rapoport-Zink uniformization theorem \cite{RZ96} allows us to describe basic loci in Shimura varieties in terms of Rapoport-Zink spaces, whose underlying reduced spaces are special cases of minuscule and basic affine Deligne-Lusztig varieties.

Let $F$ be a non-archimedean local field with finite residue field $\F_q$ of prime characteristic $p$, and let $L$ be the completion of the maximal unramified extension of $F$.
Let $\sigma$ denote the Frobenius automorphism of $L/F$.
Further, we write $\cO$ (resp.\ $\cO_F$) for the valuation ring of $L$ (resp.\ $F$).
Finally, we denote by $\vp$ a uniformizer of $F$ (and $L$) and by $v_L$ the valuation of $L$ such that $v_L(\vp)=1$.

Let $G$ be an unramified connected reductive group over $\cO_F$.
Let $B\subset G$ be a Borel subgroup and $T\subset B$ a maximal torus in $B$, both defined over $\cO_F$.
Let $X^*(T)$ (resp.\ $X_*(T)$) denote the set of characters (resp.\  cocharacters).
For $\mu\in X_*(T)$, let $\vp^{\mu}$ be the image of $\vp\in \mathbb G_m(F)$ under the homomorphism $\mu\colon\mathbb G_m\rightarrow T$.

Set $K=G(\cO)$.
There exist $\Fq$-ind schemes called the loop group $LG$, the positive loop group $L^+G$ and the affine Grassmannian $\cG r\coloneqq LG/L^+G$ of $G$ whose $\aFq$-valued points are canonically identified with $G(L)$, $K$ and $G(L)/K$.
In the equal characteristic case, $\cG r$ is representable as an inductive limit of projective varieties, while in the mixed characteristic case it is representable as an inductive limit of perfections of projective varieties (cf.\ \cite{PR08}, \cite{Zhu17} and \cite{BS17}).

We fix a dominant cocharacter $\mu\in X_*(T)_+$ and $b\in G(L)$.
Then the {\it affine Deligne-Lusztig variety} $X_{\mu}(b)$ is the locally closed reduced $\aFq$-subscheme of $\cG r$ whose $\aFq$-valued points, often identified with $X_{\mu}(b)$ itself, are
$$X_{\mu}(b)(\aFq)=\{xK\in \cG r(\aFq)\mid x^{-1}b\sigma(x)\in K\vp^{\mu}K\}.$$
Then $X_{\mu}(b)$ is locally of finite type in the equal characteristic case and locally perfectly of finite type in the mixed characteristic case (cf.\ \cite[Corollary 6.5]{HV11}, \cite[Lemma 1.1]{HV18}).
The variety $X_{\mu}(b)$ carries a natural action (by left multiplication) by the $\sigma$-centralizer of $b$
$$\J_b=\{j\in G(L)\mid j^{-1}b\sigma(j)=b\}.$$

In \cite{GH15}, G\"{o}rtz and He defined a notion of Coxeter type, and proved that if $(G,\mu,K)$ is of Coxeter type, then $X_\mu(\tau)$ is naturally a union of classical Deligne-Lusztig varieties, which is called the {\it Bruhat-Tits stratification}.
Here $\tau$ is (a representative in $G(L)$ of) a length $0$ element in the Iwahori-Weyl group $\tW$ of $T$ whose $\sigma$-conjugacy class in $G(L)$ is the unique basic element such that $X_\mu(\tau)\neq \emptyset$.
Via the relationship to Shimura varieties, this result for minuscule $\mu$ gives a simple description of basic loci in Shimura varieties of Coxeter type.
For example, the cases of Coxeter type include the $\GU(1,n-1)$-case by Vollaard-Wedhorn \cite{VW11}, the $\GU(2,2)$-case by Howard-Pappas \cite{HP14} and the Siegel case of genus $2$ by Katsura-Oort \cite{KO87}.
Such a simple description has many applications; the Kudla-Rapoport program \cite{KR11} and Zhang’s Arithmetic Fundamental Lemma \cite{Zhang12} are among them.

In \cite{CV18}, Chen and Viehmann introduced the {\it $\J$-stratification}, which is defined in complete generality.
The $\J$-strata are locally closed subsets of $\cG r$, defined by measuring relative positions against elements in $\J=\J_b$.
By intersecting each $\J$-stratum with $X_\mu(b)$, we obtain the $\J$-stratification of $X_\mu(b)$ (see \S\ref{J-str}).
They conjectured that the Bruhat-Tits stratification of $X_\mu(\tau)$, which exists only in the cases of Coxeter type, coincides with the $\J$-stratification.
In \cite{Gortz19}, G\"{o}rtz proved the conjecture under the assumption that $G$ is adjoint and absolutely simple.
The $\J$-stratification also generalizes other well-known stratifications such as the $a=1$ locus and the semi-module stratification of de Jong-Oort \cite{dJO00}.
On the other hand, these two papers were the only works on the $\J$-stratification in general, and very little was known about it; for instance, its parametrization and geometric properties were unknown.

More recently, Miaofen Chen and Xinwen Zhu proposed a conjecture predicting a parametrization of $\J_b\backslash \Irr X_\mu(b)$ by certain Mirkovi\'{c}-Vilonen cycles in $\cG r$, where $\Irr X_\mu(b)$ is the set of irreducible components of $X_\mu(b)$ (note that $X_\mu(b)$ is now known to be equidimensional, cf.\ \cite{HV12}, \cite{Takaya25}).
This conjecture was proved in special cases by Xiao-Zhu \cite{XZ17} and Hamacher-Viehmann \cite{HV18}, and in full generality by Nie \cite{Nie22} and, independently, by Zhou-Zhu \cite{ZZ20} in its numerical version.
In \cite{XZ17}, the authors used the Chen-Zhu conjecture to study the cohomology of certain Shimura varieties.

In the case where $\mu$ is minuscule and $b=\tau$, Nie also gave an explicit construction of irreducible components in \cite[Theorem 0.11]{Nie22}.
Let $\Phi\subset X^*(T)$ denote the set of roots of $T$, and let $\Pi\subseteq \Phi$ denote the set of minus simple roots and highest positive roots.
For $\alpha\in \Phi$, set $\ld_\alpha=\la\alpha,\ld\ra\in \Z$ if $\alpha$ is negative, and $\ld_\alpha=\la\alpha,\ld\ra-1\in \Z$ if $\alpha$ is positive, where $\la,\ra\colon X^*(T)\times X_*(T)\rightarrow \Z$ is the natural perfect pairing.
We say $\ld\in \Y$ is {\it small} if for each $\alpha\in \Pi$, there exists $\beta\in \cO_\alpha$ such that $\ld_\beta\le 0$, where $\cO_\alpha$ is the $p(\tau\sigma)$-orbit of $\alpha$.
Nie proved that the $\J_\tau$-orbit of each irreducible component has a representative of the form $\overline{X_\mu^\ld(\tau)}$, where $X_\mu^\ld(\tau)=X_\mu(\tau)\cap I\vp^\ld K/K$ and $I$ is the standard Iwahori subgroup associated to $T\subset B\subset G$ (cf.\ \S\ref{notation}).
He also proved that if $\dim X_\mu^\ld(\tau)=\dim X_\mu(\tau)$, then $X_\mu^\ld(\tau)$ is irreducible if and only if $\ld$ is small.
Moreover, the stabilizer in $\J_\tau$ of the irreducible component $\overline{X_\mu^\ld(\tau)}$ can be described by $\ld$.
In fact, it is easy to see that small cocharacters characterize the irreducibility of $X_\mu^\ld(\tau)$ in general (Proposition \ref{irr=small}).
The main purpose of this paper is to provide a way to study $\J$-strata in terms of small $\ld$ in the minuscule and basic case.

Let $\cS_{\mu,\tau}$ denote the set of non-empty $\J$-strata of $X_\mu(\tau)$.
By the definition of $\J$-strata, $\J_\tau$ acts on $\cS_{\mu,\tau}$, and each $X_\mu^\ld(\tau)$ is a union of $\J$-strata.
We say that a $\J$-stratum is {\it parahoric} (resp.\ {\it standard parahoric}) if its stabilizer in $\J_\tau$ is parahoric (resp.\ standard parahoric).
If $X_\mu^\ld(\tau)$ is irreducible (equivalently, if $\ld$ is small), then there exists a unique $S\in\cS_{\mu,\tau}$ that is open in $X_\mu^\ld(\tau)$.
Since $(I\cap \J_\tau)X_\mu^\ld(\tau)=X_\mu^\ld(\tau)$, we have $(I\cap \J_\tau)S=S$, i.e., $S$ is standard parahoric.
The first main theorem below states that every standard parahoric $\J$-stratum arises in this way.

Let $\cS_{\mu,\tau}^{\pa}$ denote the subset of $\cS_{\mu,\tau}$ consisting of parahoric $\J$-strata.
Then $\J_\tau$ acts on $\cS_{\mu,\tau}^{\pa}$, and each $\J_\tau$-orbit contains a (not necessarily unique) standard parahoric representative.
Let $W_0$ be the finite Weyl group of $G$.
Let $\tW$ be the Iwahori-Weyl group, and let $\Omega\subset \tW$ be the set of length $0$ elements (cf.\ \S\ref{notation}).
By \cite[Proposition 2.9]{Nie22}, $X_\mu^\ld(\tau)\neq \emptyset$ if and only if $\ld^\natural\coloneqq\tau\sigma(\ld)-\ld\in W_0\mu$.
Let $\cA_{\mu,\tau}^{\sm}=\{\ld\in\Y\mid \ld^\natural\in W_0\mu,\ \text{$\ld$ is small}\}$ (cf.\ \S\ref{small}). 
We set $\Omega\cap \J_\tau=\{w\in \Omega\mid \sigma(w)=\tau\sigma(w)\tau^{-1}=w\}$.
Note that the elements of $\Omega$ normalize $I$.
Thus $\Omega\cap \J_\tau$ acts on $\cA_{\mu,\tau}^{\sm}$.
Until the end of \S\ref{introduction}, we assume that the (absolute) root system of $G$ is irreducible, $\mu$ is minuscule and $b=\tau$.
\begin{thm}[cf.\ Theorem \ref{main theo}]
\label{main thm}
If a standard parahoric $\J$-stratum $S\in \cS_{\mu,\tau}^{\pa}$ is contained in $X_\mu^\ld(\tau)$ for some $\ld\in \Y$, then $\ld\in \cA_{\mu,\tau}^{\sm}$ and $S$ is the unique such stratum in $X_\mu^\ld(\tau)$.
Moreover, if $S'\subseteq X_\mu^{\ld'}(\tau)$ is another standard parahoric $\J$-stratum in the same $\J_\tau$-orbit as $S$, then $\ld'\in (\Omega\cap \J_\tau)\ld$.
Consequently, this induces a bijection $$\J_\tau\backslash \cS_{\mu,\tau}^{\pa}\xrightarrow{\sim} (\Omega\cap \J_\tau)\backslash\cA_{\mu,\tau}^{\sm},\quad \J_\tau S\mapsto (\Omega\cap \J_\tau)\ld.$$
\end{thm}

Define $R(\ld)\coloneqq \{\alpha\in \Phi\mid \langle\alpha, \ld^\natural\rangle=-1, \ld_\alpha\geq 1\}$ and $\Pi(\ld)\coloneqq\{\alpha\in \Pi\mid \text{$\ld_\beta\geq 0$ for all $\beta\in \cO_\alpha$}\}$.
Combining Theorem \ref{main thm} with the geometry of $X_\mu^\ld(\tau)$, we obtain the following, which generalizes \cite[Theorem 0.11]{Nie22}.
\begin{thm}[cf.\ Corollary \ref{main coro}]
\label{main thm2}
In the situation of Theorem \ref{main thm}, $S$ is open in $X_\mu^\ld(\tau)$.
In particular, $S$ is (the perfection of) an irreducible smooth quasi-affine variety.
Moreover, $\dim S=\dim X_\mu^\ld(\tau)=\sharp R(\ld)$, and the stabilizer of $S$ in $\J_\tau$ coincides with that of $\overline{X_\mu^\ld(\tau)}$; both are equal to the standard parahoric subgroup of type $\Pi(\ld)$.
\end{thm}

The information encoded by small $\ld$ is very useful.
For example, using this, Nie proved for general $(G,\mu,b)$ that the stabilizer of each irreducible component of $X_\mu(b)$ in $\J_b$ is a parahoric subgroup of maximal volume (cf.\ \cite[Theorem 0.12]{Nie22}).
Moreover, a construction of irreducible components due to Xiao-Zhu, similar to that of Nie, allowed Fox, Howard and Imai to explicitly describe the irreducible components of the basic locus of the $\GU(2,n-2)$ Shimura variety at an inert prime (cf.\ \cite{FI21},\cite{FHI26}).
In the special case where the $\J$-stratification coincides with the semi-module stratification, Viehmann \cite{Viehmann08} proved that each $\J$-stratum is isomorphic to an affine space using small $\ld$, which is essentially the same as semi-modules.

Set $\cA_{\mu,\tau}^{\tp}\coloneqq\{\ld\in \cA_{\mu,\tau}\mid \dim X_\mu^\ld(\tau)=\dim X_\mu(\tau)\}$ and $\cA_{\mu,\tau}^{\sm,\tp}\coloneqq \cA_{\mu,\tau}^{\sm}\cap \cA_{\mu,\tau}^{\tp}$.
When $\mu$ is minuscule and $b=\tau$, the Chen-Zhu conjecture can be rephrased, using Nie's description of irreducible components, as the existence of a natural bijection
$$(\Omega\cap \J_\tau)\backslash\cA_{\mu,\tau}^{\sm,\tp}\cong \{\nu\in W_0\mu\mid \nu^\diamond=\tilde\nu_\tau^\diamond\},$$
where $\nu^\diamond$ denotes the $\sigma$-average of $\nu$, and $\tilde \nu_\tau\in \Y/(1-\sigma)\Y$ denotes the ``best integral approximation'' of the Newton point $\nu_\tau$ of $\tau$ (cf.\ \S\ref{Chen-Zhu}).
Let $\Phi_+$ denote the set of positive roots determined by $B$.
Let $\epsilon_\ld\in W_0$ denote a Weyl group element defined in \S\ref{minuscule} as
$$\{\alpha\in \Phi\mid \ld_\alpha\geq 0\}=\el\Phi_+.$$
In \S\ref{Chen-Zhu}, we show that the map $(\Omega\cap \J_\tau)\ld\mapsto \ld^\flat\coloneqq \epsilon_{\ld}^{-1}(\tau\sigma(\ld)-\ld)$ is well-defined and coincides with the map constructed by Nie, which gives the above bijection.

As we will see in Proposition \ref{map}, the same construction yields a map
$$\flat\colon (\Omega\cap \J_\tau)\backslash\cA_{\mu,\tau}^{\sm}\rightarrow \{\nu\in W_0\mu\mid \nu^\diamond\le \tilde\nu_\tau^\diamond\},\quad (\Omega\cap \J_\tau)\ld\mapsto \ld^\flat= \epsilon_{\ld}^{-1}(\tau\sigma(\ld)-\ld).$$
The target coincides with $\{\nu\in W_0\mu\mid \nu^\diamond\le \nu_\tau\}$, where $\le$ denotes the partial order on $\Y_\Q$ such that $\nu\le \nu'$ if $\nu'-\nu$ is a non-negative linear combination of positive coroots.
By Proposition \ref{dim}, we have  $\dim X_\mu^\ld(\tau)=\la\rho,\mu+\ld^\flat\ra$, where $\rho$ denotes the half sum of positive roots.
It is plausible that this map is also bijective; indeed, we establish the numerical version as follows.
\begin{thm}[cf.\ Theorem \ref{flat bijective}]
\label{thm bijection}
We have $$\sharp\bigl((\Omega\cap \J_\tau)\backslash\cA_{\mu,\tau}^{\sm}\bigr)=\sharp\{\nu\in W_0\mu\mid \nu^\diamond\le \tilde\nu_\tau^\diamond\}.$$
\end{thm}
We have explicit formulas for the cardinality; see Table \ref{count-small} in \S\ref{numerical}.
We also prove the bijectivity of $\flat$ for $G=\GL_n$ or $\GSp_{2n}$ (cf.\ \S\ref{GLn} and \S\ref{Cn}).
In this case, we can identify small cocharacters with the admissible subsets introduced by Gorsky-Mazin-Oblomkov \cite{GMO26} in their study of compactified Jacobians for plane curve singularities.
Under this identification, the bijectivity of $\flat$ follows from \cite[Theorem 1.1(a)]{GMO26}.

If $G=\GL_n$ and the $\J$-stratification coincides with the semi-module stratification, the bijectivity of $\flat$ was conjectured by de Jong-Oort \cite[Remark 6.16]{dJO00} and proved by Hamacher-Viehmann \cite[Theorem 4.16]{HV18}.

The inclusion $\cS_{\mu,\tau}^{\pa}\subseteq \cS_{\mu,\tau}$ need not be an equality in general.
However, in the important cases such as the Bruhat-Tits stratification or the semi-module stratification, the equality holds.
The former stratification exists if and only if a nonnegative rational number called $\depth(G,\mu)$ is less than or equal to $1$ (in more general setting, this condition is equivalent to full Hodge-Newton decomposability rather than Coxeter type, cf.\ \cite{GHN19}).
The latter stratification exists if and only if $\tau$ is superbasic.
In \cite{SV25}, Schremmer and Viehmann studied the cases with $1<\depth(G,\mu)<2$.
In this range, $X_\mu(\tau)$ admits a nice description, similar to the Bruhat-Tits stratification.

In our setting, either $\depth(G,\mu)<2$ or $\tau$ is superbasic if and only if $(G,\mu)$ is weakly fully Hodge-Newton decomposable.
The latter condition was introduced by Chen-Tong \cite{CT22} in the context of $p$-adic Hodge theory.
By \cite[Theorem 3.5]{CT22}, it is equivalent to saying that the Newton stratification of the Schubert variety for $\mu$ in the $B_{\mathrm dR}^+$-affine Grassmannian is finer than the Harder-Narasimhan stratification.
Surprisingly, this condition characterizes the equality $\cS_{\mu, \tau}^{\pa}=\cS_{\mu, \tau}$ in the special fiber.
\begin{thm}[cf.\ Theorem \ref{weakly HN}]
\label{thm HN}
The following assertions are equivalent:
\begin{enumerate}[(1)]
\item  $\cS_{\mu, \tau}^{\pa}=\cS_{\mu, \tau}$;
\item  $\depth(G,\mu)<2$ or $\tau$ is superbasic;
\item $(G,\mu)$ is weakly fully Hodge-Newton decomposable.
\end{enumerate}
\end{thm}
The classification of weakly fully Hodge-Newton decomposable pairs is known by \cite[Proposition 2.14]{CT22} and \cite[Theorem 2.7]{SV25}.
In most such cases, each $\J$-stratum is known to be isomorphic to the product of a Deligne-Lusztig variety and an affine space, see \cite{Gortz19}, \cite{Shimada5} and \cite{ST24}.
We do not know whether this holds in general.

A further direction is to extend our results to arbitrary parahoric level, especially to Iwahori level. As shown in \cite[Theorem 2.10]{Gortz19}, the $\J$-stratification can be defined at Iwahori level as well.
Since, at Iwahori level, even the dimension formula for the basic locus of Shimura varieties is not known in general (and, to the best of my knowledge, no general conjectural formula has been proposed), we expect that this direction would shed new light on this problem.

In joint work with Takamatsu \cite{STinprepOort}, we apply the main results of this paper to prove a generalization of Oort's conjecture.
Let $\mathscr S_g$ be the moduli space over $\overline{\F}_p$ of $g$-dimensional principally polarized supersingular abelian varieties, and let $V\subset\mathscr S_g$ be the open locus where the geometric automorphism group is $\{\pm1\}$.
Oort's conjecture can be stated as $\operatorname{codim}_{\mathscr S_g}(\mathscr S_g\setminus V)\geq1$ for $g\geq2$.
Using the $\J$-stratification, we prove that $\operatorname{codim}_{\mathscr S_g}(\mathscr S_g\setminus V)\geq2$ for $g\geq4$ and $p\geq5$.
Before this result, Oort's conjecture was proved by Viehmann \cite{Viehmann26} and Karemaker--Yu \cite{KR26,KY26}.
Viehmann's proof uses an explicit study of the $a=1$ locus, while the proof of Karemaker and Yu in \cite{KR26} uses a stratification of the union of all supersingular Ekedahl-Oort strata.
The $a=1$ locus and this stratification are both closely related to the $\J$-stratification (cf.\ \cite[Proposition 5.11]{CV18} \& \cite[Remark 6.21]{KR26}).
We expect that the $\J$-stratification will play an important role in further generalizations of Oort's conjecture.

We also expect that the main results of this paper will have applications to problems involving the cohomology of certain Shimura varieties such as Kottwitz varieties (i.e., those considered in \cite{KO92}).
Indeed, some parts of the results are already known to be related to the cohomology of Shimura varieties.
For example, Xiao-Zhu \cite{XZ17} studied the Tate conjecture for Kottwitz varieties as an application of the Chen-Zhu conjecture.
In \cite{Kret13} and \cite{Kret15}, Kret computed the cohomology of the basic locus of some Kottwitz varieties, and certain lattice paths called {\it Dyck paths} appear in his results.
If $G$ is of type $A_n$, then the cardinality of the sets in Theorem \ref{thm bijection} is the number of corresponding Dyck paths, which is called the {\it rational Catalan number}.
It would be interesting to explore further connections with these papers.

The paper is organized as follows.
In \S\ref{preliminaries}, we introduce affine Deligne-Lusztig varieties and some basic facts.
In \S\ref{small flat}, we study small cocharacters and introduce $\flat$.
In \S\ref{J-stratification}, we introduce the $\J$-stratification and prove the first two main theorems.
The remainder of the paper is devoted to the analysis of $\ld$ and $\flat$.
In \S\ref{numerical}, we prove Theorem \ref{thm bijection} by explicitly counting the number of small cocharacters.
In \S\ref{weakly HN section}, we prove Theorem \ref{thm HN} using the results in \S\ref{HN}.

\textbf{Acknowledgments:}
The author would like to thank Miaofen Chen, Ian Gleason, Sian Nie, Felix Schremmer and Chia-Fu Yu for helpful comments.
The author would like to thank Sug Woo Shin for telling him about Kret's papers.

This work was supported by JSPS KAKENHI Grant number JP26K16963.
This paper was written during a stay at UC Berkeley, which was supported by the JSPS Overseas Research Fellowship.
The author would like to thank the university and the host Sug Woo Shin for their hospitality.

OpenAI's ChatGPT was used to assist with mathematical discussion, language, and bibliographic searches.

\section{Preliminaries}
\label{preliminaries}
In this paper, we will often identify a (perfect) variety with its set of closed points.
\subsection{Notation}
\label{notation}
Let $F$ be a non-archimedean local field with finite residue field $\F_q$ of prime characteristic $p$, and let $L$ be the completion of the maximal unramified extension of $F$.
Let $\sigma$ denote the Frobenius automorphism of $L/F$.
Further, we write $\cO$ (resp.\ $\cO_F$) for the valuation ring of $L$ (resp.\ $F$).
Finally, we denote by $\vp$ a uniformizer of $F$ and by $v_L$ the valuation of $L$ such that $v_L(\vp)=1$.

Let $G$ be a connected reductive group over $\cO_F$.
Let $B\subset G$ be a Borel subgroup and $T\subset B$ a maximal torus in $B$, both defined over $\cO_F$.
Let $\Phi=\Phi(G,T)$ denote the set of roots of $T$ in $G$.
We denote by $\Phi_+$ (resp.\ $\Phi_-$) the set of positive (resp.\ negative) roots distinguished by $B$.
Let $X^*(T)$ (resp.\ $X_*(T)$) be the set of characters (resp.\ cocharacters) of $T$, and let $X_*(T)_+$ be the set of dominant cocharacters.
For a cocharacter $\ld\in X_*(T)$, let $\vp^{\ld}$ be the image of $\vp\in \mathbb G_m(F)$ under the homomorphism $\ld\colon\mathbb G_m\rightarrow T$.
The Frobenius map of $G$ induces an automorphism on the based root datum of $G$ associated to $T\subseteq B\subseteq G$.
We will also denote it by $\sigma$.
There is a natural perfect pairing $\la,\ra\colon X^*(T)\times X_*(T)\rightarrow \Z$.

The {\it Iwahori-Weyl group} $\tW=\tW_G$ is defined as the quotient $N_{G(L)}T(L)/T(\cO)$.
This can be identified with the semi-direct product $W_0\ltimes X_{*}(T)$, where $W_0$ is the {\it finite Weyl group} of $G$.
Let $\bS\subset W_0$ denote the set of simple reflections.
We denote the projection $\tW\rightarrow W_0$ by $p$.
We have a length function $\ell\colon \tW\rightarrow \Z_{\geq 0}$ given as
$$\ell(\vp^{\lambda}u)=\sum_{\alpha\in \Phi_+, u^{-1}\alpha\in \Phi_-}|\langle \alpha, \lambda\rangle-1|+\sum_{\alpha\in \Phi_+, u^{-1}\alpha\in \Phi_+}|\langle \alpha, \lambda\rangle|,$$
where $u\in W_0$ and $\lambda\in \Y$.
We may and do embed $\tW$ into the group of affine transformations of $\Y_{\R}$ so that the action of $w=\vp^\ld u$ is given by $v\mapsto uv+\ld$.
For $\alpha\in \Phi$, we denote by $s_\alpha$ the reflection which sends $\ld\in \Y$ to $\ld-\langle \alpha, \ld\rangle\alpha^\vee$, where $\alpha^\vee$ is the corresponding coroot of $\alpha$.
Then $\bS=\{s_\alpha\mid \text{$\alpha$ is a simple root}\}$.

Let $\tP=\Phi\times \Z$ be the set of affine roots.
We view $a=(\alpha,k)\in \tP$ as an affine function such that $a(v)=-\la \alpha,v\ra+k$ for $v\in \Y_\R$.
Let $s_a=\vp^{k\alpha^\vee}s_\alpha$ denote the corresponding affine reflection.
If $w=\vp^\ld u\in \tW$, then $w(\alpha,k)=(u\alpha,k+\la u\alpha,\ld\ra)$.
Set $\tilde\Phi_+=\{(\alpha,k)\in \tilde\Phi\mid k\geq 1\}\sqcup \{(\alpha,0)\in \tilde\Phi\mid \alpha\in \Phi_-\}$.
Then $\tilde\Phi=\tilde\Phi_+\sqcup \tilde\Phi_-$ with $\tilde\Phi_-=-\tilde\Phi_+$.
For $\alpha\in \Phi$, we define $\tilde \alpha=(\alpha,0)\in \tP$ if $\alpha\in \Phi_-$ and $\tilde \alpha=(\alpha,1)\in \tP$ if $\alpha\in \Phi_+$.
Let $\Pi$ denote the set of minus simple roots and highest positive roots of $\Phi$.
Then $\ell(s_a)=1$ if and only if $a=\pm\ta$ for $\alpha\in \Pi$.
See \cite[\S1.2]{Nie22} for this description.

Let $\tS=\{s_{\ta}\mid \alpha\in \Pi\}$.
Let $W_a\subseteq \tW$ denote the affine Weyl group $\Z\Phi^\vee\rtimes W_0$, where $\Phi^\vee$ is the set of coroots.
Then $(W_a,\tS)$ is a Coxeter system.
Moreover, we can write the Iwahori-Weyl group as a semi-direct product $\tW=W_a\rtimes \Omega$, where $\Omega\subset \tW$ is the subgroup of length $0$ elements.
We extend the length function and the Bruhat order on the Coxeter group $W_a$ to $\tW$ in a natural way.
Then the resulting length function coincides with the function $\ell$ above.
Note also that for $\alpha\in \Pi$, $\ell(s_{\ta}\vp^\ld)=\ell(\vp^\ld)+1$ if and only if $\la\alpha,\ld\ra\le 0$.

Let $M \supseteq T$ be a semistandard Levi subgroup of $G$. 
By replacing the triple $T \subseteq B \subseteq G$ with $T \subseteq B \cap M \subseteq M$, we can define $\Phi^{M}_{\pm}$, $W_0^{M}$, $\Omega^{M}$ and so on as above. 

For $\alpha\in \Phi$, we denote by $U_\alpha\colon \cO\rightarrow G(L),z\mapsto U_\alpha(z)$ the corresponding one-parameter root subgroup.
We set $$I=T(\cO)\prod_{\alpha\in \Phi_+}U_{\alpha}(\vp\cO)\prod_{\beta\in \Phi_-}U_{\beta}(\cO)\subseteq G(L),$$
which is called the {\it standard Iwahori subgroup} associated to $T\subset B\subset G$.
We have
$$G(L)=\bigsqcup_{w\in\tW} IwI,$$
which is called the {\it Iwahori-Bruhat decomposition}.
We also have an explicit formula on the multiplication of Iwahori-Bruhat cells for $s\in \tS$ and $w\in \tW$:
$$IsIwI=\begin{cases}
IswI & \text{if $\ell(sw)=\ell(w)+1$,}\\
IswI\sqcup IwI & \text{if $\ell(sw)=\ell(w)-1$.}
\end{cases}$$

\subsection{Affine Deligne-Lusztig Varieties}
\label{ADLV}
For $\mu\in \Y_+$ and $b\in G(L)$, the affine Deligne-Lusztig variety $X_{\mu}(b)$ in the affine Grassmannian $\cG r=\cG r_G$ is the locally closed reduced $\aFq$-subscheme of $\cG r$ whose $\aFq$-valued points are
$$X_{\mu}(b)(\aFq)=\{xK\in \cG r(\aFq)=G(L)/K\mid x^{-1}b\sigma(x)\in K\vp^{\mu}K\}.$$
Note that if $\mu$ is minuscule, then $X_\mu(b)$ is closed in $\cG r$.
In the equal characteristic case, affine Deligne-Lusztig varieties are schemes, locally of finite type over $\aFq$.
In the mixed characteristic case, affine Deligne-Lusztig varieties are perfect schemes, locally perfectly of finite type over $\aFq$.
See \cite{PR08}, \cite{Zhu17}, \cite{BS17} and \cite[Lemma 1.1]{HV18}.
Left multiplication by $g^{-1}\in G(L)$ induces an isomorphism between affine Deligne-Lusztig varieties corresponding to $b$ and $g^{-1}b\sigma(g)$.
Thus the isomorphism class of the affine Deligne-Lusztig variety only depends on the $\sigma$-conjugacy class of $b$.

The affine Deligne-Lusztig varieties carry a natural action (by left multiplication) by the $\sigma$-centralizer of $b$
$$\J_b=\{j\in G(L)\mid j^{-1}b\sigma(j)=b\}.$$
Note that $\J_b\cong \J_{g^{-1}b\sigma(g)}$ by sending $j$ to $g^{-1}jg$.
In fact, $\J_b$ is the set of $F$-valued points of an algebraic group $J_b$ over $F$.
The element $b$, or its $\sigma$-conjugacy class, is called basic if $J_b$ is
an inner form of $G$ (see \cite[Definition 1.10]{Gortz19}).

The non-emptiness criterion and the dimension formula are already known for the affine Deligne-Lusztig varieties in the affine Grassmannian (see \cite{Gashi10}, \cite{Viehmann06} and \cite{Hamacher15}).
Let $B(G)$ denote the set of $\sigma$-conjugacy classes of $G(L)$. 
Thanks to Kottwitz \cite{Kottwitz85}, a $\sigma$-conjugacy class $[b]\in B(G)$ is uniquely determined by two invariants: the Kottwitz point $\kappa(b)\in \pi_1(G)/((1-\sigma)\pi_1(G))$ and the Newton point $\nu_b\in X_*(T)_{\Q,+}^\sigma$.
Then $b$ is basic if and only if $\nu_b$ is central.
Set $B(G,\mu)=\{[b]\in B(G)\mid \kappa(b)=\kappa(\vp^\mu), \nu_b\le \mu^{\diamond}\}$, where $\mu^\diamond$ denotes the $\sigma$-average of $\mu$, and $\le$ denotes the partial order on $\Y_\Q$ such that $\nu\le \nu'$ if $\nu'-\nu$ is a non-negative linear combination of positive coroots.
Then $X_\mu(b)\neq \emptyset$ if and only if $[b]\in B(G,\mu)$.
If this is the case, we have
$$\dim X_\mu(b)=\la\rho, \mu-\nu_b\ra-\tfrac{1}{2}\de(b),$$
where $\rho$ is the half sum of positive roots and $\de(b)=\mathrm{rk}_F(G)-\mathrm{rk}_F(J_b)$.

In this paper, we often consider the decomposition
$$G(L)/K=\bigsqcup_{\ld\in \Y}I\vp^\ld K/K,\quad X_\mu(b)=\bigsqcup_{\ld\in \Y}\Xl,$$
where $X^\ld_\mu(b)\coloneqq X_\mu(b)\cap I\vp^\ld K/K$.
Each piece is a locally closed subset of $X_\mu(b)$.

Let $\cF l$ denote the affine flag variety whose $\aFq$-valued points are canonically identified with $G(L)/I$.
Note that $\overline{IwI/I}=\bigsqcup_{w'\le w}Iw'I/I$.
Thus if $\ell(xw)=\ell(x)+\ell(w)$, then $xIwI/I=x(\overline{IwI/I})\cap IxwI/I$ is closed in $IxwI/I$.
Since the natural projection $\pi\colon G(L)/I\rightarrow G(L)/K$ is proper, $\pi(xIwI/I)=xIwK/K$ is closed in $\pi(IxwI/I)=IxwK/K$.
In particular, we have obtained the following lemma.
\begin{lemm}
\label{closed}
Let $x\in \tW$.
Let $\ld\in \Y$ such that $\ell(x\vp^\ld)=\ell(x)+\ell(\vp^\ld)$.
Then $x I\vp^\ld K/K\subseteq I\vp^{x\ld} K/K$ is closed in $I\vp^{x\ld} K/K$.
\end{lemm}

Let $m_{\ld}$ denote the unique minimal-length element in $\vp^\ld W_0$.
Then $I\vp^\ld K/K=Im_\ld K/K$.
Since $\pi$ is proper and $Im_\ld I/I$ is closed in $\pi^{-1}(Im_\ld K/K)$, the restriction of $\pi$ to $Im_\ld I/I$ is also proper.
Since $I\cap m_\ld Im_\ld^{-1}=I\cap m_\ld K m_\ld^{-1}$, this restriction is (universally) bijective, and hence a universal homeomorphism.
By $Im_\ld I/I\cong \A^{\ell(m_\ld)}$ (after perfection), we obtain the following lemma.
\begin{lemm}
\label{affine space}
The orbit $I\vp^\ld K/K$ is universally homeomorphic to an affine space.
\end{lemm}

\subsection{The Minuscule and Basic Case}
\label{minuscule}
In this subsection, we summarize the results in \cite[\S1.5 \& \S2.2]{Nie22}.
Let $\ld\in \Y$ and $\alpha\in \Phi$.
We set $\ld_\alpha=-\tilde\alpha(\ld)$, i.e.,
$$\lambda_\alpha=\begin{cases}
\langle\alpha, \lambda\rangle-1 & \alpha\in \Phi_+, \\
\langle\alpha, \lambda\rangle & \alpha\in \Phi_-.
\end{cases}$$
Then $\ld_\alpha+\ld_{-\alpha}=-1$ and $s_{\ta}\ld=\ld-\ld_\alpha \alpha^\vee$ (cf.\ \cite[Lemma 1.3 (1) \& (2)]{Nie22}).
Let $U_\lambda$ be the subgroup of $G$ generated by $U_\alpha$ such that $\lambda_\alpha\geq 0$.
Then there exists a unique element $\ep_\lambda$ such that $U_\ld=\el U\el^{-1}$.
Here $U$ denotes the unipotent radical of $B$.
By \cite[Lemma 1.3 (3)]{Nie22}, we have $\ld_\alpha=(\tau\ld)_{p(\tau)(\alpha)}$ and $\ep_{\tau\ld}=p(\tau)\el$ for $\tau\in \Omega$.

We assume that $\mu\in \Y_+$ is minuscule and $b$ is basic.
We may always choose $b\in N_{G(L)}T(L)$ to be a lift of an element of $\Omega$, and we denote its image in $\tW$ again by $b$.
We always assume that $[b]\in B(G,\mu)$.
We define $\ld^\natural\coloneqq b\sigma(\ld)-\ld$ and
\begin{align*}
\cA_{\mu,b}\coloneqq \{\ld\in \Y\mid X^\ld_\mu(b)\neq \emptyset\},\quad R(\ld)\coloneqq \{\alpha\in \Phi\mid \langle\alpha, \ld^\natural\rangle=-1, \ld_\alpha\geq 1\}.
\end{align*}

\begin{prop}
\label{Xl}
We have $\cA_{\mu,b}=\{\ld\in \Y\mid \ld^\natural\in W_0 \mu\}$.
If $\ld\in \cA_{\mu,b}$, then $\Xl$ is (the perfection of) a smooth variety of dimension $\sharp R(\ld)$.
Moreover, $I\cap \J_b$ acts transitively on the set of irreducible components of $\Xl$.
\end{prop}
\begin{proof}
This is \cite[Proposition 2.9]{Nie22}.
Note that the proof there implies that $\Xl$ admits a smooth deperfection (see also the proof of \cite[Proposition 1.23]{Zhu17}).
\end{proof}

Let $p\colon \tW\rtimes \langle \sigma \rangle\rightarrow W_0\rtimes \langle \sigma \rangle$ be the natural projection.
For any subset $D\subseteq \tW$, we set $D\cap \J_b=\{w\in D\mid b\sigma(w)b^{-1}=w\}$.
By abuse of notation, we also denote by $w\in \tW\cap \J_b$ some lift of $w$ in $N_{G(L)}T(L)$ that lies in $\J_b$.
For $\alpha\in \Phi$ and $i\in \Z$, we define $\alpha^i=p(b\sigma)^i(\alpha)\in \Phi$ and $\tilde \alpha^i=(b\sigma)^i(\tilde \alpha)\in \tilde \Phi$.
\begin{lemm}
\label{ld}
Let $\ld\in \Y$.
Then we have (1) $\la \alpha,\ld^\natural\ra=\ld_{\alpha^{-1}}-\ld_\alpha$ for $\alpha \in \Phi$, and
(2) $(w\ld)^\natural=p(w)\ld^\nt$ for $w\in \tW\cap \J_b$.
If, moreover, $\ld \in \cA_{\mu,b}$, then (3) $R(\tau\ld)=p(\tau)R(\ld)$ for $\tau\in\Omega\cap \J_b$.
\end{lemm}
\begin{proof}
This is \cite[Lemma 2.7 \& Lemma 2.8]{Nie22}.
\end{proof}
We set $\ld^\flat\coloneqq \el^{-1}(\ld^\natural)$.
If $\tau\in\Omega\cap \J_b$, then $(\tau\ld)^\flat=\ld^\flat$ by Lemma \ref{ld}.

\section{Small Cocharacters}
\label{small flat}
We introduce small cocharacters and the map $\flat$.
In this section, we assume that $\sigma$ acts transitively on the connected components of the Dynkin diagram of $\bS$, $\mu$ is minuscule and $[b]\in B(G,\mu)$ is basic.
Let $d$ denote the number of connected components of $\bS$.
Recall that $b$ is a length $0$ element.
\subsection{Smallness and Irreducibility}
\label{small}
We recall some useful facts from \cite[\S5]{Nie22}. 
Although Nie assumed that $G$ is adjoint, this assumption can be removed (see also \S\ref{ad}).

We say that a subset $D\subseteq \Phi$ is strongly orthogonal if $\beta\pm \gamma \notin \Phi\cup\{0\}$ for any distinct $\beta,\gamma\in D$.
If $D$ is strongly orthogonal, then it is orthogonal, i.e.,
$\langle \gamma,\beta^\vee\rangle =0$.
Then the simple affine reflections $s_{\tb}$ for $\beta\in D$ commute with each other.

For $\alpha\in \Phi$, we set $\cO_\alpha=\{\alpha^i\mid i\in \Z\}$ and $\cO_{\tilde\alpha}=\{\tilde \alpha^i\mid i\in \Z\}$, where $\alpha^i$ and $\tilde \alpha^i$ are as in \S\ref{Xl}.
Let $\Wa$ be the subgroup of $\tW$ generated by $s_{\tilde \alpha^i}$ for $\alpha^i \in \cO_\alpha$.
Recall that $\Pi\subseteq\Phi$ is the set of minus simple roots and highest positive roots.
Then $\J_b$ is generated by $I\cap \J_b, \Omega\cap \J_b$ and $\Wa\cap \J_b$ for $\alpha\in \Pi$.
\begin{lemm}
\label{Oa}
Let $\alpha\in \Pi$.
Then $\Wa$ is infinite if and only if $\cO_\alpha=\Pi$.
If $\Oa\subsetneq \Pi$, then $\Wa\cap \J_b=\{1,\wa\}$, where $\wa$ is the longest element in $\Wa$.
Moreover, one of the following cases occurs:
\begin{enumerate}[(1)]
\item $\cO_\alpha$ is strongly orthogonal and $\wa=\prod_{\beta\in \cO_\alpha}s_{\tb}$;
\item $\cO_{\alpha+\alpha^d}$ is strongly orthogonal, $\la\alpha^d,\alpha^\vee\ra=-1, |\cO_\alpha|=2d$ and $\wa=\prod_{\beta\in \cO_{\alpha+\alpha^d}}s_{\tb}$.
\end{enumerate}
\end{lemm}
\begin{proof}
This is \cite[Lemma 5.1 \& Lemma 5.3]{Nie22}.
\end{proof}

\begin{lemm}
\label{not small}
Let $\alpha\in \Pi$ and $\ld\in \Y$.
Assume that $\Oa\neq \Pi$.
Then $\{\ld_\beta\mid \beta\in \Oa\}=\{-(\wa \ld)_\beta\mid \beta\in \Oa\}$.
If, moreover, $\ld\in \cAm$ with $|\ld_\beta|\geq 1$ for $\beta\in \Oa$, then $R(\wa\ld)=p(\wa)R(\ld)$.
\end{lemm}
\begin{proof}
This is \cite[Lemma 5.6 \& Lemma 5.7]{Nie22}.
\end{proof}

Let $\alpha,\beta\in \Pi$ such that $\alpha+\beta\in \Phi$.
Then at least one of $\alpha$ and $\beta$ is a minus simple root, and we have $\widetilde{(\alpha+\beta)}=\ta+\tb$ and $\ld_{\alpha+\beta}=\ld_\alpha+\ld_\beta$ for $\ld\in \Y$.
The following lemma is an analogue of \cite[Lemma 5.8 (3)]{Nie22}.
\begin{lemm}
\label{ab}
Let $\alpha,\beta\in \Pi$ such that $\alpha+\beta\in \Phi$.
Assume that $\la\beta,\alpha^\vee\ra=\la\alpha,\beta^\vee\ra=-1$, and that either $\ld_\alpha,\ld_\beta\geq 0$ or $\ld_\alpha,\ld_\beta\le 0$.
Let $\ld'\in\{w\ld\mid w\le s_{\tb}s_{\ta}s_{\tb}=s_{\ta}s_{\tb}s_{\ta}\}$.
If either $\ld'_\alpha,\ld'_\beta\geq 0$ or $\ld'_\alpha,\ld'_\beta\le 0$, then $\ld'=\ld$ or $s_{\ta}s_{\tb}s_{\ta}\ld$.
\end{lemm}
\begin{proof}
If $\ld'=\ld,s_{\ta}\ld,s_{\tb}\ld,s_{\tb}s_{\ta}\ld,s_{\ta}s_{\tb}\ld,s_{\ta}s_{\tb}s_{\ta}\ld$, then
\begin{align*}
(\ld'_\alpha,\ld'_\beta)=
&(\ld_\alpha,\ld_\beta),\
(-\ld_\alpha,\ld_\alpha+\ld_\beta),\
(\ld_\alpha+\ld_\beta,-\ld_\beta),\\
&(\ld_\beta,-\ld_\alpha-\ld_\beta),\
(-\ld_\alpha-\ld_\beta,\ld_\alpha),\
(-\ld_\beta,-\ld_\alpha),
\end{align*}
respectively.
The assertion immediately follows from this by noticing that $s_{\ta}\ld=\ld$ if $\ld_\alpha=0$, and $s_{\tb}\ld=\ld$ if $\ld_\beta=0$.
\end{proof}

We say $\ld\in \cAm$ is {\it small} if for each $\alpha\in \Pi$, there exists some $\beta\in \cO_\alpha$ such that $\ld_\beta\le 0$ (cf.\ \cite[\S 5.4]{Nie21}).
Note that this definition coincides with the one in \cite{Nie22} if $\dim \Xl=\dim X_\mu(b)$.
Let $N_{\J_b}(\overline{\Xl})$ denote the stabilizer of $\overline{\Xl}$ in $\J_b$.
For $w\in \tW \cap \J_b$, we also denote by $w$ some lift of it in $\J_b$ (such a choice always exists by Lang's theorem for $T(\cO)$, see \cite[Proposition 3]{Greenberg63}).
\begin{prop}
\label{irr=small}
Let $\ld\in \cAm$.
Then $\Xl$ is irreducible if and only if $\ld$ is small.
\end{prop}
\begin{proof}
If $\ld$ is small, then the irreducibility of $\Xl$ follows from \cite[\S 5.6]{Nie21}.
Note that the proof literally works even if $\dim \Xl<\dim X_\mu(b)$.

For the converse, we argue by contradiction following \cite[\S6.5]{Nie22}.
Suppose that $\Xl$ is irreducible and $\ld\in \cAm$ is not small.
Then there exists $\alpha\in \Pi$ such that $\ld_\beta\geq 1$ for $\beta\in \Oa$.
By Lemma \ref{Oa}, $\Wa$ is finite because $\{\beta\in \Pi\mid\langle \beta, \ld\rangle\geq \ld_\beta\geq 1\}\neq \Pi$.
Thus we have $\wa\ld\neq \ld$ and $\dim X^{\wa\ld}_\mu(b)=\dim \Xl$ by Proposition \ref{Xl} and Lemma \ref{not small}.
Since $(\wa\ld)_\beta\le -1$ for $\beta \in \Oa$, we have $\ell(s_{\tb}\vp^{\wa\ld})=\ell(\vp^{\wa\ld})+1$.
Thus $\ell(\wa \vp^{\wa \ld})=\ell(\wa)+\ell(\vp^{\wa\ld})$ and $\wa I\vp^{\wa\ld}K\subseteq I\wa\vp^{\wa\ld}K=I\vp^\ld K$.
Hence $\wa X_\mu^{\wa \ld}(b)\subseteq \Xl$ and $\wa \overline{X_\mu^{\wa \ld}(b)}=\overline{\Xl}$.
In particular, we have $\wa N_{\J_b}(\overline{X_\mu^{\wa\ld}(b)})\wa^{-1}=N_{\J_b}(\overline\Xl)$.
Since both $N_{\J_b}(\overline{X_\mu^{\wa\ld}(b)})$ and $N_{\J_b}(\overline\Xl)$ are standard parahoric subgroups containing $I\cap \J_b$, we have $\wa\in N_{\J_b}(\overline{X_\mu^{\wa\ld}(b)})=N_{\J_b}(\overline\Xl)$, which is a contradiction.
This finishes the proof.
\end{proof}

Let $C\subseteq\Z$.
We say $C$ is {\it consecutive} if there exist $z,z'\in \Z$ with $C=[z,z']\cap\Z$.
\begin{lemm}
\label{consecutive}
Let $\ld\in \cA_{\mu,b}$ and $\alpha\in \Phi$.
The set $\{\ld_\beta\mid \beta\in \cO_\alpha\}$ is consecutive.
\end{lemm}
\begin{proof}
As $\mu$ is minuscule, this follows from Proposition \ref{Xl} and Lemma \ref{ld} (1).
\end{proof}

Note that $\Omega\cap \J_b$ acts on both $\cAm$ and $\cAm^{\sm}\coloneqq\{\ld\in \cAm\mid \text{$\ld$ is small}\}$.
Although $(\Omega\cap \J_b)\backslash \cA_{\mu,b}$ is not finite in most cases, we have the following (see also \S\ref{numerical}):
\begin{lemm}
\label{small finite}
The set $(\Omega\cap \J_b)\backslash \cA_{\mu,b}^{\sm}$ is finite.
\end{lemm}
\begin{proof}
Let $\alpha\in \Pi$.
If $\ld\in \cAm$, the set $\{\ld_\beta\mid \beta\in \cO_\alpha\}$ is consecutive by Lemma \ref{consecutive}.
If $\ld\in \cAm^{\sm}$, we have $\max\{\ld_\beta\mid \beta\in \cO_\alpha\}\le \sharp\cO_\alpha$.
Since $\alpha$ is arbitrary, there exist sufficiently large $M_1,M_2>0$ such that for any $\ld\in \cAm^{\sm}$ and $\beta\in \Phi_-$, we have $-M_2\le \la\beta,\ld\ra\le M_1$.
Thus $\cA_{\mu,b}^{\sm}$ is finite modulo central cocharacters.
Note that if $\tau\in \Omega$ satisfies $\tau \ld\in \cAm$, then $\tau\in \Omega^\sigma=\Omega\cap \J_b$.
Since central cocharacters lie in $\Omega$, this finishes the proof.
\end{proof}

\subsection{The map $\flat$}
\label{flat}

Let $h$ be the common Coxeter number of the irreducible components of $\Phi$.
Set $$\a_\ld\coloneqq \ep_\ld^{-1}(h\ld -\rho^\vee)\in \Y_\Q,\quad w_\ld\coloneqq\ep_\ld^{-1}p(b)\sigma(\ep_\ld)\in W_0,$$
where $\rho^\vee$ is the half sum of positive coroots.
Write $b=\vp^\eta p(b)$ and set $v=\ep_\ld(\a_\ld-w_\ld\sigma(\a_\ld)+h\ld^\flat)$.
Then it is straightforward to check that $v=p(b)\rho^\vee-\rho^\vee+h\eta$.
Note that $b$ acts on $\{\tilde\alpha\mid \alpha\in\Pi\}$.
Fix a minus simple root $\alpha$.
If $p(b)^{-1}(\alpha)$ is a minus simple root, then $\la\alpha,\eta\ra=0$.
Hence
$$\la\alpha,v\ra=\la p(b)^{-1}(\alpha),\rho^\vee\ra-\la\alpha,\rho^\vee\ra+h\la\alpha,\eta\ra=-1+1+0=0.$$
If $p(b)^{-1}(\alpha)$ is a highest positive root, then $\la\alpha,\eta\ra=-1$.
Hence
$$\la\alpha,v\ra=\la p(b)^{-1}(\alpha),\rho^\vee\ra-\la\alpha,\rho^\vee\ra+h\la\alpha,\eta\ra=(h-1)+1-h=0.$$
Since $\alpha$ is arbitrary, $v$ is central.

\begin{lemm}
\label{equation}
Let $\ld\in \Y$.
We have $\a_\ld^\diamond-(w_\ld\sigma(\a_\ld))^\diamond+h\ld^{\flat,\diamond}=h\nu_b$.
\end{lemm}
\begin{proof}
Recall that $\nu_b=\frac{1}{N}\sum_{i=0}^{N-1}p(b\sigma)^i\eta$ for some $N\in \N$ such that $p(b\sigma)^N=1$ (cf.\ \cite[\S1.7]{He14}).
Since $\sigma$ preserves central cocharacters, we have
$$v^\diamond=\frac{1}{N} \sum_{i=0}^{N-1} \sigma^i (v)=\frac{1}{N} \sum_{i=0}^{N-1} p(b\sigma)^i (v) = \frac{1}{N} \sum_{i=0}^{N-1} p(b\sigma)^i (p(b\sigma)\rho^\vee-\rho^\vee) + h \nu_b=h\nu_b.$$
Since $v=\el^{-1}v=\a_\ld-w_\ld\sigma(\a_\ld)+h\ld^\flat$, this finishes the proof.
\end{proof}

\begin{lemm}
\label{strictly dominant}
Let $\ld\in \Y$.
Then $\a_\ld$ is strictly dominant.
\end{lemm}
\begin{proof}
Let $\alpha\in \Phi_+$.
Set $\beta=\el \alpha$.
If $\beta\in \Phi_-$, then $\ld_\beta=\la\beta,\ld\ra\geq 0$.
Let $c_\beta>0$ denote the sum of the coefficients of $\beta$ when it is written as a linear combination of minus simple roots.
Then
$$\la\alpha,\a_\ld\ra=h\la\beta,\ld\ra-\la\beta,\rho^\vee\ra=h\la\beta,\ld\ra+c_\beta>0.$$
If $\beta\in \Phi_+$, then $\ld_\beta=\la\beta,\ld\ra-1\geq 0$.
Let $c_\beta>0$ denote the sum of the coefficients of $\beta$ when it is written as a linear combination of simple roots.
Then
$$\la\alpha,\a_\ld\ra=h\la\beta,\ld\ra-\la\beta,\rho^\vee\ra=h\la\beta,\ld\ra-c_\beta\geq h-c_\beta>0.$$
Thus in both cases, we have $\la\alpha,\a_\ld\ra>0$ as desired.
\end{proof}

\begin{prop}
\label{map}
There exists a map $$\flat\colon(\Omega\cap \J_b)\backslash\cAm\rightarrow\{\nu\in W_0\mu\mid \nu^\diamond\le \nu_b\},\quad (\Omega\cap \J_b)\ld\mapsto \ld^\flat.$$
\end{prop}
\begin{proof}
By the results in \S\ref{minuscule}, it remains to show that $\ld^{\flat,\diamond}\le \nu_b$.
This follows from Lemma \ref{equation} and Lemma \ref{strictly dominant}.
Indeed, $\sigma$ preserves $\Phi_+^\vee$ and
$$
\a_\ld^\diamond-(w_\ld\sigma(\a_\ld))^\diamond=\frac{1}{N} \sum_{i=0}^{N-1}(\sigma^{i+1}(\a_\ld)-\sigma^i(w_\ld)\sigma^{i+1}(\a_\ld)),
$$
where $N$ is the order of $\sigma$.
Thus $h(\nu_b-\ld^{\flat,\diamond})\geq 0$ as desired.
\end{proof}

We next deduce a dimension formula of $\Xl$ for $\ld\in \cAm$.
Recall that $\rho$ denotes the half sum of positive roots of $G$.
\begin{lemm}
\label{w}
\label{lwlambda}
Let $\ld\in \cAm$ and $\alpha\in \Phi_+$.
Then $(w_\ld\sigma)^{-1}\alpha\in \Phi_-$ if and only if $\la \alpha,\a_\ld\ra<h$ and $\la \alpha,\ld^\flat\ra=-1$.
Moreover, we have $\ell(w_\ld)=-2\la\rho,\ld^\flat\ra$.
\end{lemm}

\begin{proof}
Since $v$ is central, for every $\alpha\in\Phi$ we have
$$\la\alpha,w_\ld\sigma(\a_\ld)\ra=\la\alpha,\a_\ld\ra+h\la\alpha,\ld^\flat\ra\quad\text{and}\quad
\la\alpha,\a_\ld\ra=\la w_\ld\sigma(\alpha),\a_\ld\ra+h\la w_\ld\sigma(\alpha),\ld^\flat\ra.$$
The first assertion follows from the first identity, since $\sigma(\a_\ld)$ is strictly dominant by Lemma \ref{strictly dominant} and $\ld^\flat$ is minuscule.
If $\alpha\in\Phi_+$ and $w_\ld\sigma(\alpha)\in\Phi_-$, the second identity gives $0<\la\alpha,\a_\ld\ra<h$.
Consequently,
\begin{align*}
\la2\rho,\ld^\flat\ra
&=\sum_{\alpha\in\Phi_+}\left(
\left\lfloor\frac{\la(w_\ld\sigma)^{-1}\alpha,\a_\ld\ra}{h}\right\rfloor
-\left\lfloor\frac{\la\alpha,\a_\ld\ra}{h}\right\rfloor\right)\\
&=\sum_{\substack{\alpha\in\Phi_+\\w_\ld\sigma(\alpha)\in\Phi_-}}\left(
\left\lfloor-\frac{\la\alpha,\a_\ld\ra}{h}\right\rfloor
-\left\lfloor\frac{\la\alpha,\a_\ld\ra}{h}\right\rfloor\right)\\
&=-\ell(w_\ld).
\end{align*}
The second equality follows by cancellation, and the last equality follows from the preceding inequalities and $\sigma(\Phi_+)=\Phi_+$.
\end{proof}

\begin{prop}
\label{dim}
Let $\ld\in \cAm$.
We have $\dim \Xl=\la\rho,\mu+\ld^\flat\ra$.
\end{prop}
\begin{proof}
Let $\alpha\in \Phi_+$ and $\beta=\el \alpha$.
By the proof of Lemma \ref{strictly dominant}, $\la\alpha,\a_\ld\ra>h$ if and only if $\ld_\beta\geq 1$.
By Proposition \ref{Xl} and Lemma \ref{w}, we have
\begin{align*}
\dim \Xl&=\sharp\{\alpha\in \Phi_+\mid \la\alpha,\ld^\flat\ra=-1, \la\alpha,\a_\ld\ra>h\}\\
&=\sharp\{\alpha\in \Phi_+\mid \la\alpha,\ld^\flat\ra=-1\}-\sharp\{\alpha\in \Phi_+\mid (w_\ld\sigma)^{-1}\alpha\in \Phi_-\}\\
&=\la\rho,\mu-\ld^\flat\ra-\ell(w_\ld)\\
&=\la\rho,\mu+\ld^\flat\ra.
\end{align*}
Indeed, if $w\in W_0$ is the unique minimal length element such that $\ld^\flat=w\mu$, then
$$\sharp\{\alpha\in \Phi_+\mid \la\alpha,\ld^\flat\ra=-1\}=\ell(w)=\la\rho,\mu-\ld^\flat\ra.$$
The first equality follows from the equivalence of $\la\alpha,\ld^\flat\ra=-1$ and $w^{-1}\alpha\in \Phi_-$ for $\alpha\in \Phi_+$, and the second equality is straightforward.
This finishes the proof.
\end{proof}

\begin{rema}
In the superbasic case (cf.\ \S\ref{Chen-Zhu}), Nie defined a similar Weyl group element as $w_\ld$ in \cite[\S3.3]{Nie22}.
Our $w_\ld$ is the inverse of the Weyl group element, and the above results are analogues of \cite[Lemma 3.8 \& Lemma 3.10]{Nie22}.
\end{rema}

\subsection{The Chen-Zhu Conjecture}
\label{Chen-Zhu}
Set $\cAm^{\tp}\coloneqq\{\ld\in \cAm\mid \dim X_\mu^\ld(b)=\dim X_\mu(b)\}$ and $\cAm^{\sm,\tp}\coloneqq \cAm^{\sm}\cap\cAm^{\tp}$.
Nie proved the following description; see \cite[Theorem 0.11 \& Proposition 6.8]{Nie22}.
\begin{theo}
\label{sm-irr}
There exists a bijection
$$(\Omega\cap \J_b)\backslash \cAm^{\sm,\tp}\xrightarrow{\sim}\J_b\backslash \Irr X_\mu(b),\quad (\Omega\cap \J_b)\ld\mapsto \J_b\overline{\Xl}.$$
\end{theo}

Under the partial order $\le$, there is a unique maximal element $\tnu_b$ in the set
$$\{\nu\in \Y/(1-\sigma)\Y\mid \kappa(\vp^\nu)=\kappa(b),\nu^\diamond\le \nu_b\},$$
which is the so-called ``best integral approximation'' of $\nu_b$ (cf.\ \cite[\S2.1]{HV18}).
By \cite[Proposition 3.9]{Schremmer22}, we have $\de(b)=\la\nu_b,2\rho\ra-\la\tnu_b, 2\rho\ra$ and $\dim X_\mu(b)=\la\rho, \mu+\tnu_b\ra$.

We next show that the restriction of the map $\flat$ to $(\Omega\cap \J_b)\backslash\cAm^{\sm,\tp}$ coincides with the map constructed by Nie to prove the Chen-Zhu conjecture.
We say $b\in G(L)$ is {\it superbasic} if none of its $\sigma$-conjugates is contained in a proper Levi subgroup of $G$.
In particular, $b$ is basic in $G(L)$ as $\nu_b$ must be central.
In this subsection, we assume that $G$ is adjoint.
Nie constructed the map by reducing to the superbasic case.
We summarize the construction below from \cite[\S5]{Nie22}.

Set $V=\Y_\R$ and $V^{p(b\sigma)} = \{ v \in V\mid p(b\sigma)(v) = v \}$.
Define
$$V_{\gen}^{p(b\sigma)} = \{ v \in V^{p(b\sigma)} \mid \langle \alpha, v \rangle = 0 \Leftrightarrow \langle \alpha, V^{p(b\sigma)} \rangle = 0, \forall \alpha \in \Phi \},$$
which is open dense in $V^{p(b\sigma)}$.
Then $V_{\gen}^{p(b\sigma)} \cap \Y\neq\emptyset$.
Let $M_b\supseteq T$ be the Levi subgroup of $G$ with root system $\{ \alpha \in \Phi \mid \langle \alpha, V^{p(b\sigma)} \rangle = 0 \}$.
By definition, for any $v\in V_{\gen}^{p(b\sigma)}$, the centralizer $M_v$ of $v$ in $G$ coincides with $M_b$.

Fix $v \in V_{\gen}^{p(b\sigma)} \cap \Y$. 
Denote by $\overline{v}$ the unique dominant $W_0$-conjugate of $v$. Let $z$ be the minimal element of $W_0$ such that $z(v) = \overline{v}$.
Let $\Phi^{N_v}=\{\alpha\in \Phi\mid \la\alpha,v\ra>0\}$ and let $N_v = \prod_{\alpha \in \Phi^{N_v}} U_{\alpha}$.
Set $M = M_{\overline{v}} =z M_v z^{-1}= z M_b z^{-1}$ and $b^M = z b \sigma(z)^{-1}$.
Then $b^M$ is a lift of some element in $\Omega^M$, and is superbasic in $M(L)$.
By \cite[Lemma 2.4]{HV18}, the best integral approximation $\tilde \nu_b$ of $\nu_b$ in $G$ coincides with that of $\nu_{b^M}$ in $M$.
Moreover, we have $z\Phi^{M_b}_+=\Phi^M_+$ and hence $\ld_\alpha=(z\ld)_{z\alpha}$ for $\alpha\in \Phi^{M_b}$.

Let $\cAm^{\tp}(v)$ denote the set of $\ld\in \cAm^{\tp}$ such that $\ld_\alpha\geq 0$ for $\alpha\in\Phi^{N_v}$.
We may choose $v$ so that $\cAm^{\tp}(v)$ contains a small cocharacter.
By \cite[Proposition 5.12 \& Corollary 6.7]{Nie22}, for any $\ld\in \cAm^{\sm,\tp}$, we have $(\Omega\cap \J_b)\ld\cap \cAm^{\sm,\tp}(v)\neq \emptyset$.
We always assume that $\ld\in \cAm^{\sm,\tp}(v)$.
Then the map
$$(\Omega \cap \J_b)\backslash\cAm^{\sm,\tp}\rightarrow \{\nu\in W_0\mu\mid \nu^\diamond=\tilde\nu_b^\diamond\},\quad (\Omega\cap \J_b)\ld\mapsto (\epsilon^{M}_{z\ld})^{-1}(b^M\sigma(z\ld)-z\ld)$$
is well-defined and bijective (cf.\ the proof of \cite[Lemma 5.16]{Nie22}), verifying the Chen-Zhu conjecture in the minuscule and basic case.
Here $\epsilon_{z\ld}^M\in W_0^M$ is defined as in \S\ref{minuscule} by setting $G=M$.
Note that if $b$ is superbasic in $G(L)$, then this map coincides with $\flat$ (cf.\ \cite[Remark 3.4]{Nie22}).
We also refer to \cite[\S4.5 \& \S5.7]{Nie22}.
\begin{theo}[cf.\ \cite{XZ17},\cite{HV18},\cite{Nie22},\cite{ZZ20}]
\label{Chen-Zhu bijection}
We have $\epsilon^{M}_{z\ld}=z\epsilon_\ld$ for $\ld\in \cAm^{\tp}(v)$.
As a consequence, the map $\flat$ gives a bijection
$$(\Omega \cap \J_b)\backslash\cAm^{\sm,\tp}\xrightarrow{\sim} \{\nu\in W_0\mu\mid \nu^\diamond=\tilde\nu_b^\diamond\},\quad (\Omega\cap \J_b)\ld\mapsto \ld^\flat.$$
\end{theo}
\begin{proof}
The last assertion follows from $\epsilon^{M}_{z\ld}=z\epsilon_\ld$ and the Chen-Zhu conjecture, since
$$(\epsilon^{M}_{z\ld})^{-1}(b^M\sigma(z\ld)-z\ld)=\el^{-1}z^{-1}(zb\sigma(\ld)-z\ld)=\el^{-1}(b\sigma(\ld)-\ld)=\ld^\flat.$$

We claim that $\el^{-1}v=\overline{v}$.
Otherwise, there exists $\alpha\in \Phi_+$ such that $\la\alpha,\el^{-1}v\ra\le -1$.
Then $\la -\el \alpha,v\ra\geq 1$.
Since $\ld\in \cAm^{\tp}(v)$, this implies that $\ld_{-\el\alpha}\geq 0$, which is a contradiction.
This proves the claim.

By the claim, we conclude that $z\el\overline{v}=\overline{v}$, i.e., $z\el\in W_0^M$ and
$$\alpha\in \Phi^M\Leftrightarrow \la\alpha,\el^{-1}v\ra=0\Leftrightarrow \la\el\alpha,v\ra=0\Leftrightarrow \el\alpha\in \Phi^{M_b}.$$
Since $(z\ld)_{z\el\alpha}=\ld_{\el\alpha}\geq 0$, we have $\epsilon^{M}_{z\ld}=z\epsilon_\ld$, which completes the proof.
\end{proof}

\section{The $\J$-stratification}
\label{J-stratification}
We introduce a slight variant of the $\J$-stratification defined in \cite{CV18}.
We retain the name $\J$-stratification, as the two stratifications coincide in important cases.
\subsection{Definition}
\label{J-str}
For any $g,h\in G(L)$, let $\inv(g,h)$ denote the relative position, i.e., the unique element in $\tW$ such that $g^{-1}h\in I\inv(g,h)I$.
In \cite[Theorem 2.10]{Gortz19}, G\"{o}rtz proved the following finiteness property.
\begin{theo}
\label{finiteness}
Let $Y\subset G(L)/I$ be a quasi-compact subscheme of the affine flag variety.
There exist only finitely many families of the form $(\inv(j,g))_{j\in \J_b}$ with $g\in Y$.
\end{theo}

Let $\inv_K(g,h)$ denote the image of $\inv(g,h)$ under the natural projection $\tW\rightarrow \tW/W_0$.
By sending $\vp^\ld u\in \tW$ to $\ld\in \Y$, we identify $\tW/W_0$ with $\Y$.
We assign to an element $gK\in G(L)/K$ the function
$$f=f_g\colon \J_b\rightarrow \Y,\quad j\mapsto \inv_K(j,g).$$
Note that $f$ is constant on cosets $j(I\cap \J_b)$.
By Theorem \ref{finiteness}, the set $$S_f\coloneqq \{hK\in G(L)/K\mid f_h=f\}=\bigcap_{j\in \J_b}jI\vp^{f(j)}K/K$$ defines a locally closed reduced $\aFq$-subscheme of $\cG r$.
Hence this invariant induces a stratification of $\cG r$.
By definition, each $I\vp^\ld K/K$ is a union of $S_f$.
By intersecting $S_f$ with $X_\mu(b)$, we obtain the $\J_b$-stratification of $X_\mu(b)$.
We usually fix $b$ and omit it from the notation.

To define the $\J$-stratification, Chen-Viehmann in \cite{CV18} used the image of $\inv(j,g)$ in $W_0\backslash\tW/W_0\cong \Y_+$ instead of $\tW/W_0\cong \Y$.
Although our stratification is in general finer than the original stratification, they coincide in many important cases.
Indeed, if $G=\mathrm{Res}_{F'/F}\GL_n$ for some finite unramified extension $F'$ of $F$ and $b\in \Omega$, then our invariant coincides with theirs by the fact that for any $\ld,\ld'\in \Y$, $$\ld=\ld'\iff K\inv(\tau,\vp^\ld)K=K\inv(\tau,\vp^{\ld'})K\text{ for all $\tau\in \Omega\cap \J_b\cong \Z$.}$$
In particular, the semi-module stratification of de Jong-Oort coincides with our stratification, and the $a=1$ locus in the basic case is equal to the $\J_b$-orbit of a single $\J$-stratum in our sense (cf.\ \cite[Proposition 3.4 \& Proposition 5.11]{CV18}).
If $(G,\mu,K)$ is of Coxeter type, then $X_\mu(b)$ admits the Bruhat-Tits stratification (cf.\ \cite{GH15}).
The main theorem of \cite{Gortz19} shows that if $G$ is an absolutely simple group of adjoint type, then the Bruhat-Tits stratification coincides with the $\J$-stratification in the sense of Chen-Viehmann.
Let $\pi$ denote the natural projection $G(L)/I\rightarrow G(L)/K$.
By \cite[\S3.3]{Gortz19}, each Bruhat-Tits stratum has the form $\pi(Y)$ with $Y\subset G(L)/I$ such that $\inv(j,-)$ is constant on $Y$.
This implies that $\pi(Y)$ is also a $\J$-stratum in our sense.

Recall that for $g\in G(L)$, we have isomorphisms 
\begin{align*}
X_\mu(b)\cong X_\mu(g^{-1}b\sigma(g)),\quad xK\mapsto g^{-1}xK\text{\quad and\quad}
\J_b\cong \J_{g^{-1}b\sigma(g)},\quad j\mapsto g^{-1}jg.
\end{align*}
As explained in \cite[Remark 2.1]{CV18}, the $\J$-stratification heavily depends on the choice of $b$ in its $\sigma$-conjugacy class in the sense that a $\J_b$-stratum does not necessarily map to a $\J_{g^{-1}b\sigma(g)}$-stratum.
For example, $gK$ always maps to $K$, which itself forms a single $\J_{g^{-1}b\sigma(g)}$-stratum. 
On the other hand, for $w\in \tW$, the $\J_{\dot w}$-stratification is independent of the choice of lift $\dot w$ in $N_{G(L)}T(L)$ because any two lifts of $w$ in $N_{G(L)}T(L)$ are $T(\cO)$-$\sigma$-conjugate (cf.\  \cite[Lemma 2.5]{Gortz19}).
It is also pointed out in \cite[Remark 2.1]{CV18} that if $[b]$ is a basic class in $B(G, \mu)$, then a reasonable choice of $b$ is a length $0$ element.
If $b$ is such a choice, then by \cite[Theorem 3.7]{He14}, any other reasonable choice can be written as $\tau^{-1}b\sigma(\tau)$ for some $\tau\in \Omega$.
Clearly, left multiplication by $\tau^{-1}$ induces an isomorphism of $\cG r$ that maps a $\J_b$-stratum to a $\J_{\tau^{-1}b\sigma(\tau)}$-stratum.
This is not evident in the original definition by Chen-Viehmann.

\subsection{Demazure Product}
\label{Demazure}
Let $\le$ denote the Bruhat order on $\tW$.
By \cite[Lemma 1]{He09b}, for any $x,y\in \tW$, the sets $\{x'y'\mid x'\le x, y'\le y\}$, $\{xy'\mid y'\le y\}$ and $\{x'y\mid x'\le x\}$ each contain a unique maximal element, and these maximal elements coincide.
We denote this common element by $x\ast y$ and we call it the {\it Demazure product} of $x$ and $y$.
Moreover, $(\tW,\ast)$ is a monoid, and the Demazure product is determined by
the following two rules:
\begin{itemize}
\item $x\ast y=xy$ if $x,y\in \tW$ such that $\ell(xy)=\ell(x)+\ell(y)$;
\item $s\ast y=y$ if $s\in \tS,\ y\in \tW$ such that $sy<y$.
\end{itemize}
Using the Demazure product, we may write $\overline{IxIyI}=\overline{I(x\ast y)I}$ (cf.\ \S\ref{notation}).
For $x\in \tW$ and $\ld\in \Y$, let us define $x\ast \ld\in \Y$ by $x\ast \vp^\ld\in \vp^{x\ast \ld}W_0$.
Clearly, $1\ast \ld=\ld$.
\begin{lemm}
\label{Demazure action}
Let $x,y\in \tW$ and $\ld\in \Y$.
Then $x\ast (y\ast \ld)=(x\ast y)\ast \ld$.
\end{lemm}
\begin{proof}
Set $\mu=y\ast \ld$.
Let $m_{\mu}$ denote the unique minimal element in $\vp^\mu W_0$.
Then there exists $u\in W_0$ such that $\vp^\mu=m_\mu u$ and $\ell(m_\mu u)=\ell(m_\mu)+\ell(u)$.
Thus
$$x\ast \vp^\mu=x\ast (m_\mu u)=x\ast (m_\mu \ast u)=(x\ast m_\mu)\ast u.$$
By the definition of $\ast$, we have $(x\ast m_\mu)\ast u\in \{(x\ast m_\mu)u'\mid u'\le u\}$.
Hence $x\ast m_\mu\in \vp^{x\ast \mu}W_0=\vp^{x\ast (y\ast \ld)}W_0$.
By a similar argument, this implies that $$x\ast (y\ast \vp^\ld)=x\ast (\vp^\mu v)=x\ast (m_\mu\ast v')=(x\ast m_\mu)\ast v'\in (x\ast m_\mu)W_0=\vp^{x\ast (y\ast \ld)}W_0$$
for some $v,v'\in W_0$.
Then the assertion follows from $x\ast (y\ast \vp^\ld)=(x\ast y)\ast \vp^\ld$.
\end{proof}

\subsection{$\J$-strata with Parahoric Stabilizers}
\label{J-small}
Let $\cS_{\mu,b}$ denote the set of non-empty $\J$-strata of $X_\mu(b)$.
By the definition of $\J$-strata, $\J_b$ acts on $\cS_{\mu,b}$, and each $X_\mu^\ld(b)$ is a union of $\J$-strata.
We say that a $\J$-stratum is {\it parahoric} (resp.\ {\it standard parahoric}) if it admits a parahoric (resp.\ standard parahoric) stabilizer in $\J_b$.
Let $\cS_{\mu,b}^{\pa}$ denote the subset of $\cS_{\mu,b}$ consisting of parahoric $\J$-strata.
Then $\J_b$ acts on $\cS_{\mu,b}^{\pa}$, and each $\J_b$-orbit contains a (not necessarily unique) standard parahoric representative.

In the rest of this subsection, we assume that the (absolute) root system of $G$ is irreducible, $\mu$ is minuscule and $[b]\in B(G,\mu)$ is basic.
We also assume that $b\in \Omega$.

Recall that $\J_b$ is generated by $I\cap \J_b, \Omega\cap \J_b$ and $\Wa\cap \J_b$ for $\alpha\in \Pi$.
Moreover, $W_a\cap \J_b$ is a Coxeter group whose simple affine reflections are the longest elements $\wa$ of $\Wa\cap \J_b$ for $\alpha\in \Pi$ such that $\Wa$ is finite (cf.\ \cite[\S5.5]{Nie22}).

By Lemma \ref{Oa}, one of the following cases occurs:
\begin{enumerate}[(1)]
\item $\cO_\alpha$ is strongly orthogonal and $\wa=\prod_{\beta\in \cO_\alpha}s_{\tb}$;
\item $\cO_{\alpha+\alpha^1}=\{\alpha+\alpha^1\}$, $\la\alpha^1,\alpha^\vee\ra=\la\alpha,{\alpha^1}^\vee\ra=-1$ and $\wa=s_{\tilde \alpha} s_{\tilde \alpha^1}s_{\tilde \alpha}$.
\end{enumerate}

Let $S_f$ be a $\J$-stratum of $\cG r$ corresponding to a function $f\colon \J_b\rightarrow \Y$ such that $S\coloneqq S_f\cap X_\mu(b)\in \cS_{\mu,b}$.
Then $S$ is standard parahoric if and only if $(I\cap \J_b) S=S$.
This means that if $j\in (I\cap \J_b)w(I\cap \J_b)$ for some $w\in \tW\cap \J_b$, then $f(j)=f(w)$.
In other words, $f$ is constant on each double coset in $(I\cap \J_b)\backslash\J_b/(I\cap \J_b)$.
\begin{prop}
\label{unique S}
Let $S$ be as above, and let $\ld\in\cAm$ such that $S\subseteq X_\mu^\ld(b)$.
Assume that $S$ is standard parahoric.
Then $f$ is uniquely determined by $f(j)=f(w)=w^{-1}\ast \ld$, where $\inv(1,j)=w$.
In particular, $S$ is the unique standard parahoric $\J$-stratum in $X_\mu^\ld(b)$.
\end{prop}
\begin{proof}
Clearly, we may assume that $w\in W_a\cap \J_b$.
Let $w=w_{\alpha_1}\cdots w_{\alpha_l}$ be a reduced expression, where $\alpha_1,\ldots, \alpha_l\in \Pi$.
Then $\ell(w)=\ell(w_{\alpha_1})+\cdots +\ell(w_{\alpha_l})$ (see \cite[\S1]{Steinberg68} or \cite[\S2.2]{KR00}).
We argue by induction on $l$.
If $l=0$, then the assertion is trivially true.
Assume that the assertion is true for $l-1\geq 0$.
Set $\ld'=(w_{\alpha_{l-1}}\cdots w_{\alpha_1})\ast \ld$.
Then
$$w_{\alpha_{l-1}}\cdots w_{\alpha_1}S_f\subseteq I \vp^{\ld'}K/K,\quad w^{-1}S_f=w_{\alpha_l}w_{\alpha_{l-1}}\cdots w_{\alpha_1}S_f\subseteq w_{\alpha_l}I \vp^{\ld'}K/K.$$
By Lemma \ref{Demazure action}, we have $$w_{\alpha_l}\ast \ld'=w_{\alpha_l}\ast(w_{\alpha_{l-1}}\cdots w_{\alpha_1})\ast \ld=w_{\alpha_l}\ast w_{\alpha_{l-1}}\ast\cdots \ast w_{\alpha_1}\ast \ld=(w_{\alpha_l}\cdots w_{\alpha_1})\ast \ld.$$
Thus we need to show that $f(w)=w_{\alpha_l}\ast \ld'$, i.e., $w^{-1}S_f\subseteq I \vp^{w_{\alpha_l}\ast \ld'}K/K$.
Set $\alpha=\alpha_l$.
Let $i\in I\cap \J_b$ such that $\wa i\wa\in (I\cap \J_b)\wa(I\cap \J_b)$ and hence $wi\wa\in (I\cap \J_b)w(I\cap \J_b)$.
Such an element $i$ exists because $\wa (I\cap \J_b)\wa\neq I\cap \J_b$ and $$(I\cap \J_b)\wa(I\cap \J_b)\wa(I\cap \J_b)=(I\cap \J_b)\sqcup(I\cap \J_b)\wa(I\cap \J_b).$$

First assume that $\wa$ is in case (1) above.
In this case, the reflections $s_{\tb}$ for $\beta\in \cO_\alpha$ commute with each other because $\cO_\alpha$ is orthogonal.
Recall that $\ell(s_{\tb} \vp^{\ld'})=\ell(\vp^{\ld'})+1$ if and only if $\la\beta,\ld'\ra\le 0$.
Thus there exists a subset $\cO'_{\alpha}\subseteq \cO_{\alpha}$ such that
$$\{\beta\in \cO_{\alpha}\mid \ld'_\beta\le -1\}\subseteq \cO'_{\alpha},\quad\text{and}\quad f(w)=\wa'\ld' \quad\text{with}\quad \wa'=\prod_{\beta\in \cO'_{\alpha}}s_{\tb}.$$
We claim that $\{\beta\in \cO_{\alpha}\mid \ld'_\beta\geq 1\}\cap \cO'_{\alpha}=\emptyset$.
Indeed, if there exists $\beta\in \cO'_\alpha$ such that $\ld'_\beta\geq 1$, then $f(w)_\beta=-\ld'_\beta\le -1$.
This implies that $f(wi\wa)_\beta=-f(w)_\beta\geq 1$.
On the other hand, since $S$ is standard parahoric, we have $f(wi\wa)=f(w)$ and $f(wi\wa)_\beta=f(w)_\beta\le -1$, which is a contradiction.
This proves the claim.
Note that if $\ld'_\beta=0$, then $s_{\tb}\ld'=\ld'$.
Thus $f(w)=\wa'\ld'=\wa\ast\ld'$ by the characterization of the Demazure product in \S\ref{Demazure}.

Next assume that $\wa$ is in case (2) above.
By Lemma \ref{consecutive} and $\ld'\in \cAm$, we have $\ld'_\alpha-\ld'_{\alpha^1}\in \{0,\pm 1\}$.
Since $f(w)\in \cAm$, we also have $f(w)_\alpha-f(w)_{\alpha^1}\in \{0,\pm 1\}$.
Thus, by Lemma \ref{ab}, we have $f(w)=\ld'$ or $\wa\ld'$.
If $\ld'_\alpha=\ld'_{\alpha^1}=0$, then $f(w)=\ld'=\wa\ast\ld'$.
If $\ld'_\alpha,\ld'_{\alpha^1}\le -1$, then $\ell(\wa \vp^{\ld'})=\ell(\wa)+\ell(\vp^{\ld'})$.
Hence $f(w)=\wa\ld'=\wa\ast \ld'$.
Thus the remaining cases are either $\{\ld'_\alpha,\ld'_{\alpha^1}\}=\{0,-1\}$ or $\ld'_\alpha,\ld'_{\alpha^1}\geq 0$.
If $\{\ld'_\alpha,\ld'_{\alpha^1}\}=\{0,-1\}$, then $f(w)=\ld'$ is impossible and hence $f(w)=\wa \ld'$.
It is straightforward to see $\wa \ld'=\wa\ast \ld'$ in this case.
If $\ld'_\alpha,\ld'_{\alpha^1}\geq 0$, $(\ld'_\alpha,\ld'_{\alpha^1})\neq (0,0)$ and $f(w)=\wa\ld'(\neq \ld')$, then $\{f(w)_\alpha ,f(w)_{\alpha^1}\}=\{-\ld'_\alpha,-\ld'_{\alpha^1}\}$ by Lemma \ref{not small}.
Similarly to the above discussion, we have $f(wi\wa)=\ld'$.
Since $S$ is standard parahoric, we have $\ld'=f(wi\wa)=f(w)=\wa \ld'$, which is a contradiction.
Thus $f(w)=\ld'=\wa\ast \ld'$.
By induction, this completes the proof.
\end{proof}

\begin{lemm}
\label{covering}
Let $\ld\in\cAm$ and $\alpha\in\Pi$ such that $\ld_\beta\geq 1$ for all $\beta\in\Oa$.
Then
$$\Xl=\bigcup_{i\in I\cap\J_b}i\wa X_\mu^{\wa\ld}(b).$$
\end{lemm}
\begin{proof}
By Lemma \ref{Oa}, $\Wa$ is finite because $\{\beta\in \Pi\mid\langle \beta, \ld\rangle\geq \ld_\beta\geq 1\}\neq \Pi$.
By Proposition \ref{Xl} and Lemma \ref{not small}, we have $\dim X_\mu^{\wa \ld}(b)=\dim \Xl$ and $(\wa \ld)_\beta\le -1$ for $\beta\in \Oa$.
Thus $\wa X_\mu^{\wa \ld}(b)$ is a closed subset of $\Xl$ by Lemma \ref{closed}.
Again by Proposition \ref{Xl}, we have $\Xl=\bigcup_{i\in I\cap\J_b}i\wa X_\mu^{\wa\ld}(b)$, as desired.
\end{proof}

\begin{coro}
\label{S small}
Let $S$ be as above, and let $\ld\in\cAm$ such that $S\subseteq X_\mu^\ld(b)$.
Assume that $S$ is standard parahoric.
Then $\ld$ is small.
\end{coro}
\begin{proof}
We argue by contradiction.
Suppose that $\ld$ is not small.
Then there exists $\alpha\in \Pi$ such that $\ld_\beta\geq 1$ for $\beta\in \Oa$.
By Lemma \ref{covering}, there exists $i\in I\cap \J_b$ such that $S\cap i\wa X_\mu^{\wa \ld}(b)\neq \emptyset$.
Since $S$ is standard parahoric, we have $S\cap\wa X_\mu^{\wa \ld}(b)\neq \emptyset$.
Let $gK\in S\cap\wa X_\mu^{\wa \ld}(b)$.
Since $f_g=f$, we have $f(\wa)=f_g(\wa)=\wa \ld$.
On the other hand, by Proposition \ref{unique S}, we have $f(\wa)=\wa\ast\ld=\ld\neq \wa \ld$, which is a contradiction.
This finishes the proof.
\end{proof}

For $S\in \cS_{\mu,b}^{\pa}$, let $N_{\J_b}(S)$ denote the stabilizer of $S$ in $\J_b$.
\begin{theo}
\label{main theo}
Let $\ld\in \cAm$.
The following conditions are equivalent:
\begin{enumerate}[(1)]
\item $\ld$ is small;
\item $\Xl$ is irreducible;
\item There exists a standard parahoric $\J$-stratum $S\in \cS_{\mu,b}^{\pa}$ contained in $\Xl$;
\item There exists a unique standard parahoric $\J$-stratum $S\in \cS_{\mu,b}^{\pa}$ contained in $\Xl$.
\end{enumerate}
Moreover, if these conditions hold and $S'\in \J_b S$ is another standard parahoric $\J$-stratum, then $S'\in (\Omega\cap \J_b)S$.
Consequently, there exists a bijection $$\J_b\backslash \cS_{\mu,b}^{\pa}\xrightarrow{\sim} (\Omega\cap \J_b)\backslash\cA_{\mu,b}^{\sm},\quad \J_b S\mapsto (\Omega\cap \J_b)\ld,$$
where $\Xl$ contains a standard parahoric representative of $\J_b S$.
\end{theo}
\begin{proof}
The equivalence of (1) and (2) is Proposition \ref{irr=small}.
The equivalence of (3) and (4) follows from Proposition \ref{unique S}.
The implication $(3)\Rightarrow (1)$ is Corollary \ref{S small}.
Assume that (2) holds.
Let $S\subseteq \Xl$ be a $\J$-stratum such that $\dim S=\dim \Xl$.
By the assumption, the closure of $S$ in $\Xl$ is $\Xl$ itself.
Since $S$ is locally closed, it is open in $\Xl$.
By $(I\cap \J_b)\Xl=\Xl$ and the assumption, we have $(I\cap \J_b)S=S$.
This proves $(2)\Rightarrow (3)$.
Therefore, all the conditions are equivalent.

Assume that the conditions hold.
Let $S'\in \J_b S$ be another standard parahoric $\J$-stratum.
It remains to show that $S'\in (\Omega\cap \J_b)S$.
Since $S$ is standard parahoric, there exists $w\in \tW\cap \J_b$ such that $S'=wS$.
Then $N_{\J_b}(S')=wN_{\J_b}(S)w^{-1}$.
Since both $N_{\J_b}(S)$ and $N_{\J_b}(S')$ are standard parahoric subgroups containing $I\cap \J_b$, we have $w\in (\Omega\cap\J_b)N_{\J_b}(S)$.
This completes the proof.
\end{proof}

Set $\Pi(\ld)\coloneqq\{\alpha\in \Pi\mid \text{$\ld_\beta\geq 0$ for all $\beta\in \cO_\alpha$}\}(\neq \Pi)$.
This definition originates from \cite[\S5.4]{Nie21}.
Let $P_{\Pi(\ld)}$ denote a standard parahoric subgroup of $\J_b$ generated by $I\cap \J_b$ and $\Wa\cap \J_b$ for $\alpha\in \Pi(\ld)$. 

\begin{coro}
\label{main coro}
Assume that the conditions in Theorem \ref{main theo} hold.
Then $S$ is open in $\Xl$ and hence (the perfection of) an irreducible smooth quasi-affine variety.
Moreover, we have $\dim S=\dim X_\mu^\ld(b)=\sharp R(\ld)$ and $N_{\J_b}(S)=N_{\J_b}(\overline{\Xl})=P_{\Pi(\ld)}$.
\end{coro}

\begin{proof}
By the proof of Theorem \ref{main theo}, $S$ is open in $\Xl$ and has the same dimension as $\Xl$.
Since $\Xl$ is (the perfection of) an irreducible smooth variety of dimension $\sharp R(\ld)$ by Proposition \ref{Xl}, the same holds for $S$.
Moreover, $S$ is quasi-affine because $\Xl$ is a closed subvariety of an affine space (cf.\ Lemma \ref{affine space}), which is also affine.
Note that $\overline{S}=\overline{\Xl}$.
Clearly, we have $N_{\J_b}(S)\subseteq N_{\J_b}(\overline{\Xl})$.
If $\wa \overline{\Xl}=\overline{\Xl}$, then $\wa S$ is a $\J$-stratum that is open in the irreducible variety $\overline{\Xl}$.
This forces $S=\wa S$.
Thus $N_{\J_b}(S)= N_{\J_b}(\overline{\Xl})$.
Clearly, we have $N_{\J_b}(S)\subseteq P_{\Pi(\ld)}$ by Proposition \ref{unique S}.
If $\alpha\in \Pi(\ld)$, then we have $w_\alpha S\subseteq \Xl$ again by Proposition \ref{unique S}.
Since $S$ is the unique $\J$-stratum in $\Xl$ with $\dim S=\dim \Xl$, we conclude that $S=\wa S$.
Hence $N_{\J_b}(S)=P_{\Pi(\ld)}$, which completes the proof.
\end{proof}

\begin{rema}
\label{only case (1)}
In the proof of Proposition \ref{unique S}, the irreducibility of the root system is only used to handle case (2) of  Lemma \ref{Oa}.
The subsequent results of this section are proved using the facts in \S\ref{small}.
Thus the results of this subsection remain valid for any $G$ over $\cO_F$ for which $\sigma$ acts transitively on the connected components of the Dynkin diagram of $\bS$ and for which case (2) does not occur (e.g., $G=\mathrm{Res}_{F’/F}\GL_n$ for a finite unramified extension $F'/F$).
\end{rema}

\subsection{Top-Dimensional $\J$-strata}
\label{top}
We keep the assumption in \S\ref{J-small}.
For $\ld,\ld'\in \Y$, we write $\ld\tle \ld'$ if either $\overline{\ld} \lneq \overline{\ld'}$ or $\ld'\in W_0\ld$ and $\ld'\le \ld$, where $\overline{\ld}$ and $\overline{\ld'}$ denote the dominant $W_0$-conjugates of $\ld$ and $\ld'$ respectively.
If $\vp^\ld u\le \vp^{\ld'}u'$ for some $u,u'\in W_0$, then $\ld\tle \ld'$.
See \cite[\S2.7]{Macdonald03}.

\begin{lemm}
\label{J-str small}
For $S\in \cS_{\mu,b}$, there exist $j\in \J_b$ and $\ld\in \cAm^{\sm}$ such that $jS\subseteq \Xl$.
\end{lemm}

\begin{proof}
Let $\lambda$ be the minimal cocharacter in the set
$$\{\ld'\in \cAm\mid \exists j\in \J_b, jS\subseteq X_\mu^{\ld'}(b)\}$$
under the partial order $\tle$.
We may assume that $S\subseteq \Xl$.
Suppose that $\ld$ is not small.
Then there exists $\alpha\in \Pi$ such that $\ld_\beta\geq 1$ for $\beta\in\Oa$.
By Lemma \ref{Oa}, $\Wa$ is finite. 
By Lemma \ref{not small}, we have $\ell(\wa \vp^{\wa\ld})=\ell(\wa)+\ell(\vp^{\wa\ld})$.
Thus we have $\wa\ld\triangleleft \ld$.
By Lemma \ref{covering}, there exists $i\in I\cap \J_b$ such that $S\cap i\wa X_\mu^{\wa \ld}(b)\neq \emptyset$.
This implies that $w_\alpha i^{-1}S\subseteq X_\mu^{w_\alpha\ld}(b)$, which is a contradiction.
Hence $\ld$ is small.
\end{proof}

Let $\cS_{\mu,b}^{\tp}\coloneqq \{S\in \cS_{\mu,b}\mid \dim S=\dim X_\mu(b)\}$.
The following corollary can be seen as a variant of Theorem \ref{sm-irr}.
See also \cite[Theorem 3.1.1]{ZZ20}.
\begin{coro}
\label{top para}
We have $\cS_{\mu,b}^{\tp}\subseteq \cS_{\mu,b}^{\pa}$.
As a consequence, there exists a bijection
$$\J_b\backslash \cS_{\mu,b}^{\tp}\xrightarrow{\sim}\J_b\backslash \Irr X_\mu(b),\quad \J_b S\mapsto \J_b \overline{S}.$$
In particular, every irreducible component admits a parahoric stabilizer in $\J_b$.
\end{coro}
\begin{proof}
Let $S\in \cS_{\mu,b}^{\tp}$.
Let $j\in \J_b$ and $\ld\in \cAm^{\sm}$ such that $jS\subseteq \Xl$.
By Corollary \ref{main coro}, the unique open $\J$-stratum in $\Xl$ is standard parahoric.
Since $\dim S=\dim X_\mu(b)=\dim \Xl$, $jS$ must be the open stratum.
Thus $S\in \cS_{\mu,b}^{\pa}$ as desired.
This combined with  Corollary \ref{main coro} implies that each $S\in  \cS_{\mu,b}^{\tp}$ is irreducible.
Hence the above map is well-defined.
Since $S$ is the unique $\J$-stratum that is open in $\overline{S}$, the injectivity is clear.
Let $C\in \Irr X_\mu(b)$.
Then  there exists $S\in \cS_{\mu,b}^{\tp}$ such that $\dim (S\cap C)=\dim X_\mu(b)$.
Thus $\overline{S}=\overline{S\cap C}=C$.
This proves the surjectivity.
The proof is finished.
\end{proof}

\subsection{The Equality $\cS_{\mu,b}^{\pa}=\cS_{\mu,b}$}
\label{HN}
We keep the assumption in \S\ref{J-small}.
In this subsection, we record some basic facts that will be used in \S\ref{weakly HN section}.
We say $b\in G(L)$ is superbasic if none of its $\sigma$-conjugates is contained in a proper Levi subgroup of $G$.
In particular, $b$ is basic in $G(L)$.
If $b$ is superbasic, then $G$ must be of type $A$ by \cite[Lemma 3.1.1]{CKV15}.
By Lemma \ref{Oa}, $b$ is superbasic if and only if $\Oa=\Pi$ for some/any $\alpha\in \Pi$.
In this case, we have $\cS_{\mu,b}^{\pa}=\cS_{\mu,b}$ and $\cAm^{\sm}=\cAm$.
See \cite[\S3]{CV18}.
In this subsection, we always assume that $b$ is not superbasic and that $b$ is the unique length $0$ element in $\Omega\cap \vp^\mu W_0$.
\begin{lemm}
\label{0 parahoric}
If $\vp^\ld K$ is contained in a standard parahoric $\J$-stratum $S$, then $\vp^\ld K=\tau K$ for some $\tau\in \Omega\cap \J_b$ and $S=\{\vp^\ld K\}=\{\tau K\}$.
\end{lemm}

\begin{proof}
By Corollary \ref{S small}, $\ld$ is small.
By Proposition \ref{unique S}, we have $w\ast \ld=w\ld$ for any $w\in \tW\cap \J_b$.
In particular, we have $\ld_\alpha\le 0$ for any $\alpha\in \Pi$.
Otherwise, since $b$ is not superbasic, we have $\wa\ast \ld\neq\wa\ld$ for some $\alpha\in \Pi$, which is a contradiction.
This implies that $\ld$ is minuscule.
Then the unique element $\tau$ in $\Omega\cap \vp^\ld W_0$ belongs to $\Omega\cap \J_b$ and $\{\tau K\}$ forms a single $\J$-stratum as desired.
\end{proof}

\begin{coro}
\label{par necessity}
If $\cS_{\mu,b}^{\pa}=\cS_{\mu,b}$, then $\cAm=\{\inv_K(1,w)\mid w\in \tW\cap \J_b\}.$
\end{coro}
\begin{proof}
Assume that $\cS_{\mu,b}^{\pa}=\cS_{\mu,b}$.
Let $\ld\in \cAm$.
Let $S$ be a $\J$-stratum containing $\vp^{\lambda}K\in X_\mu(b)$, which is parahoric by assumption.
Since $T(\cO)\cap \J_b\subset N_{\J_b}(S)$, there exists $w\in W_a\cap\J_b$ such that $wS$ is standard parahoric.
Then by Lemma \ref{0 parahoric}, $w\vp^{\ld}K=\vp^{w\ld}K=\tau K$ for some $\tau\in \Omega\cap \J_b$ and it forms a $0$-dimensional $\J$-stratum.
Thus $\ld\in \{\inv_K(1,w\tau)\mid w\in W_a\cap \J_b,\tau\in \Omega\cap \J_b\}$.
Since the other inclusion is obvious, this finishes the proof.
\end{proof}

Recall that $W_a\cap \J_b$ is a Coxeter group whose simple affine reflections are the longest elements $\wa$ of $\Wa\cap \J_b$ for $\alpha\in \Pi$ such that $\Wa$ is finite.
For $J\subseteq \Pi$, let $W_J$ denote the standard parabolic subgroup of $W_a$ generated by $s_{\tilde\alpha}$ for $\alpha\in J$.
\begin{lemm}
\label{every parahoric}
For $\ld\in \cAm^{\sm}$, let $w\in W_a\cap \J_b$ be the unique minimal length representative of its coset in $(W_a\cap \J_b)/(W_{\Pi(\ld)}\cap \J_b)$.
Assume that $w\ast\ld=w\ld$ holds for all $\ld\in \cAm^{\sm}$ and all such $w$.
Then every $\J$-stratum $S$ in $X_\mu(b)$ is parahoric.
\end{lemm}
\begin{proof}
For $\ld\in \cAm^{\sm}$, let $T$ be the union of $\J$-strata $S\subset \Xl$ such that $P_{\Pi(\ld)}S\subset \Xl$.
Note that $T\neq \emptyset$ by Corollary \ref{main coro}.
If $w\ast\ld=w\ld$ for every $w$ as above, then $T$ must be equal to the unique standard parahoric $\J$-stratum in $\Xl$.

Suppose that there exists $S\notin \cS_{\mu,b}^{\pa}$.
Let  $\ld$ be the minimal cocharacter in the set
$$\{\ld'\in \cAm\mid \exists j\in \J_b, jS\subseteq X_\mu^{\ld'}(b)\}$$
under the partial order $\tle$.
Then $\ld$ is small by the proof of Lemma \ref{J-str small}.
By replacing $S$ in its $\J_b$-orbit, we may assume that $S\subset \Xl\setminus T$.
On the other hand, there exist $j\in P_{\Pi(\ld)}$ and $\ld'\triangleleft \ld$ such that $j S\subset X_\mu^{\ld'}(b)$.
This is a contradiction.
Therefore, every $\J$-stratum $S$ in $X_\mu(b)$ is parahoric as desired.
\end{proof}

\subsection{Passage to Adjoint Groups}
\label{ad}
We return to the general situation.
Let $G_{\ad}$ denote the adjoint group of $G$.
Let $K_{\ad}=G_{\ad}(\cO)$ and let $I_{\ad}$ denote the Iwahori subgroup containing the image of $I$.
Let $\Omega_{\ad}\subseteq \tW_{G_{\ad}}$ denote the subgroup of length $0$ elements.
For an element associated with $G$, we denote its image in $G_{\ad}$ by the subscript $\ad$, for example $g_{\ad},\mu_{\ad},b_{\ad},w_{\ad}$.

Let $\tau\in \Omega\cong \pi_1(G)$ and let $\cG r^\tau$ denote the corresponding connected component of $\cG r$.
Set $X_\mu(b)^\tau\coloneqq X_\mu(b)\cap \cG r^\tau$.
If $X_\mu(b)^\tau\neq \emptyset$, then by \cite[Corollary 2.4.2]{CKV15} and \cite[Proposition 3.1]{HV18}, there is a universal homeomorphism
$$X_\mu(b)^\tau\rightarrow X_{\mu_{\ad}}(b_{\ad})^{\tau_{\ad}}.$$

Until the end of this subsection, we assume that $\mu$ is minuscule, $b$ is a length $0$ element such that  $X_\mu(b)\neq \emptyset$.
Although we did not take this approach here, the results in \S\ref{J-small} can be proved by passage to adjoint groups.
For this, the following lemma is a key.

\begin{lemm}
\label{ad f}
Let $g,h\in G(L)$.
Then $gK,hK\in \cG r^\tau$ for some $\tau\in \Omega$ and $f_{g_{\ad}}=f_{h_{\ad}}$ if and only if $f_g=f_h$.
In particular, for any $w\in W_a\cap \J_b$, $f_g$ is constant on $(I\cap \J_b)w(I\cap \J_b)$ if and only if $f_{g_{\ad}}$ is constant on $(I_{\ad}\cap \J_{b_{\ad}}) w_{\ad}(I_{\ad}\cap \J_{b_{\ad}})$.
\end{lemm}
\begin{proof}
Assume that $gK,hK\in \cG r^\tau$ for some $\tau\in \Omega$ and $f_{g_{\ad}}=f_{h_{\ad}}$.
It follows from $f_{g_{\ad}}=f_{h_{\ad}}$ that for any $j\in \J_b$, the elements $\inv_K(j,g)$ and $\inv_K(j,h)$ differ by a central cocharacter.
They must coincide because $gK,hK\in \cG r^\tau$.

Assume that $f_g=f_h$.
Then $gK$ and $hK$ lie in the same connected component because $f_g(1)=f_h(1)$.
Let $w\in \J_b$ be a lift of an element in $W_a\cap \J_b=W_a\cap \J_{b_{\ad}}$.
Since $\Omega_{\ad}\cap \J_{b_{\ad}}$ normalizes $I_{\ad}$, to prove $f_{g_{\ad}}=f_{h_{\ad}}$, it suffices to show that $f_{g_{\ad}}(iw_{\ad})=f_{h_{\ad}}(iw_{\ad})$ for any $i\in I_{\ad}\cap \J_{b_{\ad}}$.
For $t\in T_{\ad}(\cO)$, we have $w_{\ad}^{-1}tw_{\ad}\in T_{\ad}(\mathcal O)$.
Moreover, by Lang's theorem for $T(\cO)$, any $i\in I_{\ad}\cap \J_{b_{\ad}}$ can be written as $i=u_{\ad}t$ for some $u\in I\cap \J_b$ and $t\in T_{\ad}(\cO)\cap \J_{b_{\ad}}$.
Thus
$$f_{g_{\ad}}(iw_{\ad})=\inv_{K_{\ad}}(u_{\ad}tw_{\ad}, g_{\ad})=\inv_{K_{\ad}}(u_{\ad}w_{\ad}, g_{\ad})=f_g(uw)_{\ad}.$$
The same holds true for $h$.
Therefore the assumption implies $f_{g_{\ad}}(iw_{\ad})=f_{h_{\ad}}(iw_{\ad})$.

The last assertion follows from the above argument, which finishes the proof.
\end{proof}

\begin{rema}
In the original definition by Chen-Viehmann, the $\J$-stratification does not behave well under passage to adjoint groups in the sense that $f_g$ does not necessarily determine $f_{g_{\ad}}$ (cf.\ \cite[Lemma 2.4 \& Remark 2.5]{CV18}).
The problem is that $\J_b\rightarrow \J_{b_{\ad}}$ is not necessarily surjective and $\Omega$ does not normalize $K$ in general.
\end{rema}

Since the connected components $\cG r^\tau$ that meet $X_\mu(b)$ differ only by $\Omega^\sigma=\Omega\cap \J_b$,
we have $X_\mu(b)=(\Omega\cap \J_b)X_\mu(b)^\tau$, where $X_\mu(b)^\tau\neq \emptyset$.
Let $\J_b^0$ denote the subgroup of $\J_b$ generated by $I\cap \J_b$ and $W_a\cap \J_b$.
Let $\cS_{\mu,b}^\tau$ (resp.\ $\cS_{\mu,b}^{\pa,\tau}$) denote the subset of $\cS_{\mu,b}$ (resp.\ $\cS_{\mu,b}^{\pa}$) consisting of the $\J$-strata contained in $X_\mu(b)^\tau$.
Then we have $$\J_b\backslash \cS_{\mu,b}\cong \J_b^0\backslash \cS_{\mu,b}^\tau,\quad \J_b\backslash \cS_{\mu,b}^{\pa}\cong \J_b^0\backslash \cS_{\mu,b}^{\pa,\tau}.$$
\begin{prop}
The projection $G\rightarrow G_{\ad}$ induces a bijective map $$\J_b\backslash \cS_{\mu,b}^{\pa}\xrightarrow{\sim} \J_{b_{\ad}}\backslash \cS_{\mu_{\ad},b_{\ad}}^{\pa}.$$
\end{prop}
\begin{proof}
By Lang's theorem for $T(\cO)$ and the Iwahori-Bruhat decomposition, any $j'\in \J_{b_{\ad}}^0$ can be written as $j'=j_{\ad}i$ with $j\in \J_b^0$ and $i\in I_{\ad}\cap \J_{b_{\ad}}$.
If $S\subset X_\mu(b)^\tau$ is a standard parahoric $\J$-stratum, then $j'S_{\ad}=j_{\ad}iS_{\ad}=j_{\ad}S_{\ad}$ by Lemma \ref{ad f}.
This implies the injectivity.
The surjectivity follows from the above universal homeomorphism and Lemma \ref{ad f}.
\end{proof}

We have analogous statements for $\cAm$:
\begin{lemm}
\label{cAm ad}
The projection $G\rightarrow G_{\ad}$ induces bijective maps
$$(\Omega\cap \J_b)\backslash \cA_{\mu,b}\cong (\Omega_{\ad}\cap \J_{b_{\ad}})\backslash \cA_{\mu_{\ad},b_{\ad}},\quad (\Omega\cap \J_b)\backslash \cA_{\mu,b}^{\sm}\cong (\Omega_{\ad}\cap \J_{b_{\ad}})\backslash \cA_{\mu_{\ad},b_{\ad}}^{\sm}.$$
\end{lemm}
\begin{proof}
The existence of the maps is clear.
By the definition of smallness, it is also clear that the first map sends $\cAm^{\sm}$ to $\cA_{\mu_{\ad},b_{\ad}}^{\sm}$.
Thus it suffices to show the first map is bijective.
This follows immediately from the above universal homeomorphism.
\end{proof}

\section{The Numerical Identity}
\label{numerical}
We assume that the (absolute) root system of $G$ is irreducible, $\mu$ is minuscule and non-central.
Let $b\in \Omega$ with $[b]\in B(G,\mu)$.
By Proposition \ref{map}, we have a map $$\flat\colon(\Omega\cap \J_b)\backslash\cAm^{\sm}\rightarrow\{\nu\in W_0\mu\mid \nu^\diamond\le \nu_b\},\quad (\Omega\cap \J_b)\ld\mapsto \ld^\flat.$$
It is easy to see that the map $\flat$ and its source and target do not depend on the choice of $b$, and are stable under passage to adjoint groups (cf.\ Lemma \ref{cAm ad}).
The results of this section can be summarized as follows.
\begin{theo}
\label{flat bijective}
We have $$\sharp\bigl((\Omega\cap \J_b)\backslash\cAm^{\sm}\bigr)=\sharp\{\nu\in W_0\mu\mid \nu^\diamond\le \nu_b\}.$$
Moreover, the map $\flat$ is bijective in the case where $G$ is of type $A_n$ or $C_n$.
\end{theo}

Let $\omega_i^\vee$ denote the fundamental coweight, 
where our labeling of simple roots follows Bourbaki \cite[Chapitre VI \S4]{Bourbaki68}.
The explicit formulas are summarized in Table \ref{count-small}.
All tables in this paper concern absolutely simple adjoint groups.

\begin{table}[htbp]
\centering
\renewcommand{\arraystretch}{1.2}
\[
\begin{array}{|c|c|c|}
\hline
\text{Type of }G & \text{Minuscule coweight }\mu & \sharp\bigl((\Omega\cap \J_b)\backslash\cAm^{\sm}\bigr)=\sharp\{\nu\in W_0\mu\mid \nu^\diamond\le \nu_b\} \\
\hline
A_n
& \omega_m^\vee,\ m=1,\ldots,n
& \mathrm{Cat}_{m,n+1-m}(\text{rational Catalan number}) \\
\hline
{}^2A_n
& \omega_m^\vee,\ m=1,\ldots,n
& \binom{\lfloor \frac{n+1}{2} \rfloor}{\lfloor \frac{m}{2} \rfloor}\binom{\lceil \frac{n+1}{2} \rceil}{\lceil \frac{m}{2} \rceil} \\
\hline
B_n
& \omega_1^\vee
& n \\
\hline
C_n
& \omega_n^\vee
& \binom{n}{\lfloor \frac{n}{2}\rfloor} \\
\hline
D_n
& \omega_1^\vee
& n-1 \\
\hline
D_n
& \omega_{n-1}^\vee,\omega_n^\vee
& \binom{n-1}{\lfloor \frac{n-2}{2}\rfloor} \\
\hline
{}^2D_n(n>4)
& \omega_1^\vee
& n+1 \\
\hline
{}^2D_n(n>4)
& \omega_{n-1}^\vee,\omega_n^\vee
& \binom{n-1}{\lfloor \frac{n-1}{2}\rfloor} \\
\hline
{}^2D_4,{}^3D_4
& \omega_1^\vee,\omega_3^\vee,\omega_4^\vee
& 3\text{ or }5 \\
\hline
E_6
& \omega_1^\vee,\omega_6^\vee
& 8 \\
\hline
{}^2E_6
& \omega_1^\vee,\omega_6^\vee
& 15 \\
\hline
E_7
& \omega_7^\vee
& 21 \\
\hline
\end{array}
\]
\caption{The number of small cocharacters}
\label{count-small}
\end{table}

In the sequel, we compute both sides case by case for each Dynkin type.
The computation of $\sharp\{\nu\in W_0\mu\mid \nu^\diamond\le \nu_b\}$ is easy.
We compute $\sharp\bigl((\Omega\cap \J_b)\backslash\cAm^{\sm}\bigr)$ by counting possible $(\ld_\alpha)_{\alpha\in \Pi}$, which determines $\ld$ completely.
It is plausible that $\flat$ is bijective in general.
For small rank $n$, this can be checked by a computer program.

\subsection{Type $A_{n-1}$}
\label{GLn}
Without loss of generality, we may assume that $G=\GL_n$.
We use the following description.
Let $T$ be the torus of diagonal matrices, and we choose the subgroup of upper triangular matrices $B$ as Borel subgroup.
Let $\chi_{ij}$ be the character $T\rightarrow \Gm$ defined by $\mathrm{diag}(t_1,t_2,\ldots, t_n)\mapsto t_i{t_j}^{-1}$.
Then we have $\Phi=\{\chi_{ij}\mid i\neq j\}$, $\Phi_+=\{\chi_{ij}\mid i< j\}$, $\Phi_-=\{\chi_{ij}\mid i> j\}$.
The set of simple roots is $\{\chi_{i,i+1}\mid 1\le i <n\}$.
Hence $\Pi=\{\chi_{i+1,i}\mid 1\le i\le n-1\}\sqcup \{\chi_{1,n}\}$.
The finite Weyl group $W_0$ is isomorphic to the symmetric group $\mathfrak S_n$.
Let us write $s_1=(1\ 2), s_2=(2\ 3), \ldots, s_{n-1}=(n-1\ n)$.
Set $s_0=\vp^{\chi_{1,n}^{\vee}}(1\ n)$.
Then $\bS=\{s_1,s_2,\ldots, s_{n-1}\}$ and $\tS=\bS\cup\{s_0\}$.

We identify $X_*(T)$ with $\Z^n$ and denote the $i$-th entry of $\ld\in X_*(T)$ by $\ld(i)$.
We set $\mu=(1^{(m)}, 0^{(n-m)}), c=s_1s_2\cdots s_{n-1},\tau=\vp^{(1,0^{(n-1)})}c$ and $b=\tau^m$ for $0< m<n$.
Then $\nu_b=(\frac{m}{n},\ldots,\frac{m}{n})$, $p(b\sigma)=c^m$ and $\tau$ is a generator of $\Omega\cong \Z$.
For $\ld\in \cAm$, we define $\a_\ld\coloneqq\el^{-1}((0,1,\ldots,n-1)+n\ld)$ (note that this definition is slightly different from the one in \S\ref{flat}).
By the definition of $\el$, $\a_\ld$ is strictly dominant, i.e., $\a_\ld(i)>\a_\ld(j)$ for $1\le i<j\le n$.

We prove the bijectivity of $\flat$ by identifying small cocharacters with the admissible invariant subsets of Gorsky-Mazin-Oblomkov \cite{GMO26} and $\flat$ with the map $\mathcal{G}$ of Gorsky-Mazin-Vazirani \cite{GMV17}.

A subset $A\subset\Z$ is called {\it $(m,n-m)$-invariant} if it is bounded below, contains every sufficiently large integer, and satisfies
$$A+m\subseteq A,\qquad A+n-m\subseteq A.$$
It is called {\it $0$-normalized} if $\min A=0$.
Set $d=\gcd(m,n)$.
For $r\in\Z/d\Z$, set $A_r=A\cap(r+d\Z)$.
Following \cite[Lemma 3.10 \& Definition 3.13]{GMO26}, we call such a subset {\it admissible} if, for every $r\in\Z/d\Z$,
$$A_{r+1}\not\subseteq A_r+n+1.$$
For $\ld\in\Y$, set
$A^\ld=\{i-1+n\ld(i)+kn\mid 1\le i\le n,\ k\in\Z_{\ge0}\}.$
\begin{lemm}
\label{demo invariant}
The map $\ld\mapsto A^\ld$ induces a bijection between $\cAm$ and the set of $(m,n-m)$-invariant subsets.
Moreover, $\ld$ is small if and only if $A^\ld$ is admissible.
Consequently, it induces a bijection
$$\Omega\backslash\cAm^{\sm}\cong
\{\text{$0$-normalized admissible $(m,n-m)$-invariant subsets}\}.$$
\end{lemm}
\begin{proof}
Set $a_i=i-1+n\ld(i)$ for $1\le i\le n$.
If $1\le j\le n$ and $j\equiv i+m\pmod n$, then
$$a_i+m-a_j=n\ln(j),\qquad
a_j+n-m-a_i=n(1-\ln(j)).$$
Thus $A^\ld$ is $(m,n-m)$-invariant if and only if $\ln(j)\ge0$ and $1-\ln(j)\ge0$ for all $1\le j\le n$.
These inequalities hold if and only if every entry of $\ln$ is $0$ or $1$, or equivalently, $\ld\in\cAm$.
Conversely, let $A$ be an $(m,n-m)$-invariant subset.
Since $A+n\subseteq A$, we have $A=A^\ld$ for the unique $\ld$ defined by
$$\ld(i)=\frac{\min(A\cap(i-1+n\Z))-(i-1)}{n}\qquad(1\le i\le n).$$
This proves the first assertion.

For $A=A^\ld$, we have $A_r=\bigcup_i(a_i+n\Z_{\ge0})$, where $i$ runs over all $1\le i\le n$ with $i-1\equiv r\pmod d$.
Thus the inclusion $A_{r+1}\subseteq A_r+n+1$ holds if and only if
$$a_{i+1}\ge a_i+n+1
\quad\text{for all }1\le i\le n\text{ with }i-1\equiv r\pmod d,$$
where $a_{n+1}=a_1$.
For $i<n$, this inequality is equivalent to $\ld_{\chi_{i+1,i}}=\ld(i+1)-\ld(i)\ge1$.
For $i=n$, it is equivalent to $\ld_{\chi_{1,n}}=\ld(1)-\ld(n)-1\ge1$.
Hence the admissibility condition is precisely the smallness condition.
The last assertion follows from $A^{\tau\ld}=A^\ld+1$, which completes the proof.
\end{proof}

For an $(m,n-m)$-invariant subset $A$, write
$$\{c_1,\ldots,c_n\}=(A\setminus(A+n-m))\sqcup((A-m)\setminus A),\qquad c_1>\cdots>c_n.$$
Following \cite[Definition 2.11 \& \S3]{GMV17} (cf.\ Remark \ref{remark sweep} below), we define
$$\mathcal{G}\colon\{\text{$(m,n-m)$-invariant subsets}\}\to\{\nu\in W_0\mu\mid\nu\le\nu_b\}$$
by
$$\mathcal{G}(A)(i)=\begin{cases}
0&\text{if }c_i\in A,\\
1&\text{if }c_i\notin A
\end{cases}\qquad(1\le i\le n).$$
By \cite[Theorem 4.10]{GMO26} and \cite[Theorem 1.1]{GMV17}, $\mathcal{G}$ induces a bijection
$$\{\text{$0$-normalized admissible $(m,n-m)$-invariant subsets}\}\xrightarrow{\sim}\{\nu\in W_0\mu\mid\nu\le\nu_b\}.$$

\begin{theo}
\label{main theo GLn}
The following map is bijective:
$$\flat\colon\Omega\backslash\cAm^{\sm}
\rightarrow\{\nu\in W_0\mu\mid\nu\le\nu_b\},
\quad\Omega\ld\mapsto\ld^\flat.$$
\end{theo}
\begin{proof}
By Lemma \ref{demo invariant} and the above bijection, it suffices to show that $\mathcal{G}(A^\ld)=\ld^\flat$.
Set $a_i=i-1+n\ld(i)$ for $1\le i\le n$. We have
$$\{a_1,\ldots,a_n\}=\{\a_\ld(1),\ldots,\a_\ld(n)\}=A^\ld\setminus(A^\ld+n),\qquad\a_\ld(1)>\cdots>\a_\ld(n).$$
By $a_i+m-a_j=n\ln(j)$ for $j\equiv i+m\pmod n$, we have
$$\ln(j)=\begin{cases}
0&\text{if }a_j-m\in A^\ld,\\
1&\text{if }a_j-m\notin A^\ld,
\end{cases}\qquad
\ld^\flat(j)=\begin{cases}
0&\text{if }\a_\ld(j)-m\in A^\ld,\\
1&\text{if }\a_\ld(j)-m\notin A^\ld.
\end{cases}$$
By \cite[Lemma 3.8]{GMV17}, we have
$$\{\a_\ld(i)-m\mid1\le i\le n\}=(A^\ld\setminus(A^\ld+n-m))\sqcup((A^\ld-m)\setminus A^\ld).$$
Thus $c_i=\a_\ld(i)-m$.
Hence $\mathcal{G}(A^\ld)=\ld^\flat$ by the definition of $\mathcal{G}$.
\end{proof}

An $(m,n-m)$-{\it Dyck path} is a lattice path from $(0,0)$ to $(n-m,m)$ with steps $(1,0)$ and $(0,1)$ that lies weakly below the diagonal.
The following pictures are examples of $(5,7)$-Dyck paths:
\begin{center}
\begin{tikzpicture}[scale=0.5]

\begin{scope}[shift={(0,0)}]
\draw[step=1cm, gray, very thin] (0,0) grid (7,5);
\draw[dashed, darkgray] (0,0) -- (7,5);
\draw[black, ultra thick]
(0,0)--(1,0)--(2,0)--(2,1)--(3,1)--(3,2)--(4,2)--(5,2)--(5,3)--(6,3)--(6,4)--(7,4)--(7,5);
\end{scope}

\begin{scope}[shift={(10,0)}]
\draw[step=1cm, gray, very thin] (0,0) grid (7,5);
\draw[dashed, darkgray] (0,0) -- (7,5);
\draw[black, ultra thick]
(0,0) -- (1,0) -- (2,0) -- (2,1) -- (3,1) -- (4,1) -- (4,2) -- (5,2) -- (6,2) -- (6,3) -- (7,3) -- (7,4) -- (7,5);
\end{scope}

\begin{scope}[shift={(20,0)}]
\draw[step=1cm, gray, very thin] (0,0) grid (7,5);
\draw[dashed, darkgray] (0,0) -- (7,5);
\draw[black, ultra thick]
(0,0)--(1,0)--(2,0)--(3,0)--(4,0)--(4,1)--(5,1)--(6,1)--(6,2)--(7,2)--(7,3)--(7,4)--(7,5);
\end{scope}

\end{tikzpicture}
\end{center}
In the case of $\GL_n$, the set $\{\nu\in W_0\mu\mid \nu\le \nu_b\}$ can be identified with the set of $(m,n-m)$-Dyck paths by reading $1$ up and $0$ right.
For example, if $n=12$ and $m=5$, the corresponding $\nu$ to the above pictures are the following respectively:
\begin{align*}
\nu&=(0,0,1,0,1,0,0,1,0,1,0,1)(=\tilde \nu_b),\\
\nu&=(0,0,1,0,0,1,0,0,1,0,1,1),\\
\nu&=(0,0,0,0,1,0,0,1,0,1,1,1).
\end{align*}

Let $\mathrm{Cat}_{m,n-m}$ be the number of $(m,n-m)$-Dyck paths, which is called the {\it rational Catalan number}.
Then $\sharp \{\nu\in W_0\mu\mid \nu\le \nu_b\}=\mathrm{Cat}_{m,n-m}$.
We have
$$\mathrm{Cat}_{m,n-m}=\begin{cases}
\frac{1}{n}\binom{n}{m} & \text{if $\gcd(m,n)=1$,}\\
\frac{1}{m+1}\binom{2m}{m} & \text{if $n=2m$.}
\end{cases}$$
Note that the latter case is the classical Catalan number.
The formula for $\mathrm{Cat}_{m,n-m}$ in other cases is much more complicated (see \cite{Bizley54}).
By Theorem \ref{main theo GLn}, $\flat$ is bijective and hence
$$\sharp(\Omega\backslash\cAm^{\sm})=\sharp\{\nu\in W_0\mu\mid \nu\le \nu_b\}=\mathrm{Cat}_{m,n-m}.$$

\begin{rema}
\label{remark sweep}
The map $\mathcal{G}$ of \cite{GMV17} takes values in the set of rational Dyck paths, which we identify with $\{\nu\in W_0\mu\mid\nu\le\nu_b\}$ as above.
The construction in \cite[\S3.3]{GMV17}, together with \cite[Theorem 4.10]{GMO26}, gives a bijective map $\mathcal{D}$ from the set of $0$-normalized admissible $(m,n-m)$-invariant subsets to the set of $(m,n-m)$-Dyck paths.
By \cite[Theorem 3.32]{GMV17}, $\mathcal{G}\circ\mathcal{D}^{-1}$ is the {\it sweep map} defined in \cite[Definition 3.1]{GMV17}.
Its bijectivity was proved by Thomas-Williams \cite[Theorem 6.7]{TW18}.
The proof of the bijectivity of $\mathcal{G}$ thus reduces to that of the sweep map.

If $\gcd(m,n)=1$, the $(m,n-m)$-invariant subsets are precisely the {\it semi-modules} for $m,n$ with type in $W_0\mu$.
The maps $\mathcal{D}$ and $\mathcal{G}$ are respectively their {\it type} and {\it cotype} maps \cite[Definition 4.9 \& Definition 4.13]{HV18}.
In particular, $\mathrm{cotype}(A^\ld)=\ld^\flat$.
In this case, Hamacher-Viehmann proved Theorem \ref{main theo GLn}.
Their proof also reduces to the bijectivity of the sweep map.
See \cite[\S4]{HV18}.
\end{rema}
\subsection{Type $C_n$}
\label{Cn}
Let $n\geq 2$.
Let us denote by $\GSp_{2n}\subset \GL_{2n}$ the group of symplectic similitudes of dimension $2n$ as in \cite[\S 2.3]{GC10}.
In the case $G=\GSp_{2n}$, we will use the following description.
Let $T$ (resp.\ $B$) be the intersection of the torus (resp.\ Borel subgroup) of $\GL_{2n}$ with $\GSp_{2n}$.
See \cite[\S 8]{GC12} for the description of the corresponding roots.
The cocharacter group $X_*(T)$ can be identified with the set $\{(m_1,\cdots, m_{2n})\in \Z^{2n}\mid m_1+m_{2n}=m_2+m_{2n-1}=\cdots=m_n+m_{n+1}\}$.
Set $s_1=(1\ 2)(2n-1\ 2n), s_2=(2\ 3)(2n-2\ 2n-1),\ldots, s_{n-1}=(n-1\ n)(n+1\ n+2), s_n=(n\ n+1)$.
Then $\bS=\{s_1,s_2,\ldots,s_n\}$ and the finite Weyl group is the subgroup of the symmetric group of degree $2n$ generated by $\bS$.
Set $s_0=\vp^{\chi_{1,2n}^{\vee}}(1\ 2n)$.
Then $\tS=\bS\cup\{s_0\}$.

Under the above identification, the positive coroots of $\GSp_{2n}$ are as follows. 
For $1\le i<j\le n$,
$$(0,\dots,\overset{i}{1},\dots,\overset{j}{-1},\dots,
\overset{2n-j+1}{1},\dots,\overset{2n-i+1}{-1},\dots,0),$$
$$(0,\dots,\overset{i}{1},\dots,\overset{j}{1},\dots,
\overset{2n-j+1}{-1},\dots,\overset{2n-i+1}{-1},\dots,0).$$
For $1\le i\le n$,
$$(0,\dots,\overset{i}{1},\dots,\overset{2n-i+1}{-1},\dots,0).$$
Thus we obtain$$\rho^\vee =(n - \tfrac{1}{2}, n - \tfrac{3}{2}, \dots, \tfrac{1}{2}, -\tfrac{1}{2}, \dots,-n + \tfrac{3}{2}, -n + \tfrac{1}{2})$$
Notice that this is identical to $\rho^\vee$ for $\GL_{2n}$.
Since the Coxeter numbers for $\GSp_{2n}$ and $\GL_{2n}$ are both $2n$, it follows from Lemma \ref{strictly dominant} that if $\ld\in \Y$, then $\epsilon_\ld$ for $\GSp_{2n}$ coincides with that for $\GL_{2n}$.

For $\mu=\omega_n^\vee=(1^{(n)},0^{(n)})$, the corresponding $b$ is $\vp^\mu(1\ n+1)\cdots (n\ 2n)$, which remains the same when viewing $\mu$ as a dominant cocharacter of $\GL_{2n}$.
Thus both $\Omega\backslash \cAm$ and $\Omega\backslash \cAm^{\sm}$ can be seen as subsets of the corresponding sets for $\GL_{2n}$.
Therefore, the map $\flat$ for $\GSp_{2n}$ can be obtained as a restriction of the corresponding map for $\GL_{2n}$.
By Theorem \ref{main theo GLn}, the restricted $\flat$ is injective.
If $\ld\in \Z^{2n}$ satisfies $b\ld-\ld\in \mathfrak S_{2n}\mu$, then it is straightforward to check that $(-w_0\ld)^\flat=(1,\ldots,1)-w_0\ld^\flat$.
If, moreover, $\ld^\flat\in W_0\mu$, then $(-w_0\ld)^\flat=\ld^\flat$.
Thus the bijectivity of $\flat$ for $\GL_{2n}$ implies that the map $\flat$ for $\GSp_{2n}$ is also bijective.

Let $\nu\in W_0\mu$.
The condition $\nu \le \nu_b$ requires $\sum_{j=1}^k \nu(j) \le \frac{k}{2}$ for all $1 \le k \le 2n$. 
By symmetry, this is equivalent to requiring $\sum_{j=1}^k \nu(j) \le \frac{k}{2}$ for $1 \le k \le n$. 
The number of such sequences with $r\le \lfloor \frac{n}{2} \rfloor$ entries equal to $1$ is given by $\binom{n}{r} - \binom{n}{r-1}$.
Indeed, if a sequence fails the condition, then there is a first position where the number of $1$'s exceeds that of $0$'s.
Since $r\le \lfloor \frac{n}{2} \rfloor$, exchanging $0$ and $1$ in the initial segment up to that position gives a bijection with sequences having $r-1$ entries equal to $1$ and $n-r+1$ entries equal to $0$.
Summing these terms over all $0 \le r \le \lfloor \frac{n}{2} \rfloor$ yields 
$$\sharp(\Omega\backslash\cAm^{\sm})=\sharp\{\nu\in W_0\mu\mid \nu\le \nu_b\}=\binom{n}{\lfloor \frac{n}{2} \rfloor}.$$

\subsection{Type ${}^2A_{n-1}$}
\label{2An}
Let $n\geq 2$.
For simplicity, we consider the Iwahori-Weyl group of $\GL_n$ and $\mu=\omega_m^\vee=(1^{(m)},0^{(n-m)})$ instead.
We may assume that $m>0$.
Let $w_0$ denote the longest element of $W_0$.
For any $\lambda\in \Y$, we have $\sigma(\ld)=-w_0\ld$.
Let $N=\lfloor \frac{n}{2}\rfloor$.

Let $\nu\in W_0\mu$.
The condition $\nu^\diamond \le \nu_b=0$ requires $\sum_{i=1}^k \nu(i) \le \sum_{i=1}^k \nu(n+1-i)$ for all $1 \le k \le N$.
The number of such sequences with exactly $r\le\lfloor \frac{j}{2} \rfloor$ entries equal to $1$ among $\{\nu(1), \dots, \nu(N)\}$ is given by $\binom{N}{r}\binom{N}{j-r} - \binom{N}{r-1}\binom{N}{j-r+1}$, where $j = \sum_{i=1}^N \bigl(\nu(i) + \nu(n+1-i)\bigr)$.
Indeed, if a sequence fails the condition, then there is a first index $k$ where the number of $1$’s among $\{\nu(1), \dots, \nu(k)\}$ exceeds the number of $1$'s among $\{\nu(n+1-k), \dots, \nu(n)\}$.
Since $r\le\lfloor \frac{j}{2} \rfloor$, exchanging the values of $\nu(i)$ and $\nu(n+1-i)$ in the initial segment up to that index $k$ gives a bijection with sequences having $r-1$ entries equal to $1$ among $\{\nu(1), \dots, \nu(N)\}$ and $j-r+1$ entries equal to $1$ among $\{\nu(n+1-N), \dots, \nu(n)\}$.
Summing these terms over all $0 \le r \le \lfloor \frac{j}{2} \rfloor$ yields $\binom{N}{\lfloor \frac{j}{2} \rfloor} \binom{N}{\lceil \frac{j}{2} \rceil}$.
If $n$ is even, $j = m$.
If $n$ is odd, $j = m$ or $m-1$ depending on $\nu(N+1) \in \{0, 1\}$.
Summing these two cases via $\binom{N}{\lceil \frac{m}{2} \rceil} + \binom{N}{\lceil \frac{m}{2} \rceil - 1} = \binom{N+1}{\lceil \frac{m}{2} \rceil}$ yields:$$\binom{N}{\lfloor \frac{m}{2} \rfloor} \binom{N}{\lceil \frac{m}{2} \rceil} + \binom{N}{\lfloor \frac{m-1}{2} \rfloor} \binom{N}{\lceil \frac{m-1}{2} \rceil} = \binom{N}{\lfloor \frac{m}{2} \rfloor} \binom{N+1}{\lceil \frac{m}{2} \rceil}.$$
Since $N = \lfloor \frac{n}{2} \rfloor$, both cases yield the formula
$$\sharp\{\nu\in W_0\mu\mid \nu^\diamond\le \nu_b\} =\binom{\lfloor \frac{n}{2} \rfloor}{\lfloor \frac{m}{2} \rfloor} \binom{\lceil \frac{n}{2} \rceil}{\lceil \frac{m}{2} \rceil}.$$

We next count the elements of $\cAm^{\sm}$.
Let $c=s_1s_2\cdots s_{n-1}$ and $b=\vp^\mu c^m\in \Omega$.
Note that $(c^m w_0)^2=1$ and $\Omega\cap \J_b=\{1\}$.
Thus, $\ld^\natural=\nu\in W_0\mu$ if and only if $c^m w_0\ld+\ld=\mu-\nu$.
For $i\in \{1,\ldots,n\}$, set $i^*=c^m w_0(i)$.
Then  $i^* \equiv m+1-i \pmod n$.
For $\alpha\in \Pi$, set $\alpha^*=p(b\sigma)\alpha=-c^m w_0\alpha\in \Pi$.
Let $M=\lfloor \frac{m}{2}\rfloor$.
Let $N'$ be the number of pairs of indices $(i, i^*)$ such that $i \neq i^*$.
The sets of fixed indices and fixed roots in $\Pi$ are explicitly described as follows:
$$
\begin{array}{l|ccc}
(n,m) & \{i \mid i=i^*\} & \{\alpha \mid \alpha=\alpha^*\} & N' \\ \hline
(2N, 2M) & \emptyset & \{\chi_{M+1, M}, \chi_{M+N+1, M+N}\} & N \\
(2N, 2M+1) & \{M+1, M+N+1\} & \emptyset & N-1 \\
(2N+1, 2M) & \{M+N+1\} & \{\chi_{M+1, M}\} & N \\
(2N+1, 2M+1) & \{M+1\} & \{\chi_{M+N+2, M+N+1}\} & N
\end{array}
$$

Note that $c^m w_0\mu=\mu$.
Thus, there exists $\ld\in \cAm$ with $\ld^\natural=\nu$ if and only if $c^m w_0\nu=\nu$ and $\mu(i)=\nu(i)$ for all $i=i^*$.
Assume that this holds.
We describe the number of small $\ld$ satisfying $\ld^\natural=\nu$ in terms of $\ld_\alpha$.
Set $i_1 = M(=n$ if $m=1$).
For $2 \le k \le N'+1$, we define $i_k \in \{1,\ldots,n\}$ by $i_k \equiv i_{k-1} - 1 \pmod n$.
Then the set $I = \{i_1, \ldots, i_{N'}\}$ consists of one index from each pair $\{i,i^*\}$ with $i\ne i^*$.
For each $1 \le k \le N'+1$, let $\alpha_k = \chi_{i_k+1, i_k}(=\chi_{1,n}$ if $i_k=n$).
Then $\Pi=\{\alpha_1,\alpha_1^*\}\sqcup\cdots \sqcup\{\alpha_{N'+1},\alpha_{N'+1}^*\}$, and the fixed roots are among $\alpha_1$ and $\alpha_{N+1}$.

Note that $\ld_{\alpha^*}-\ld_\alpha=\la \alpha,\nu\ra$ for $\alpha\in \Pi$ by Lemma \ref{ld} (1).
Note also that by $c^m w_0\nu=\nu$, we have $\la\alpha^*,\nu\ra=-\la\alpha,\nu\ra$.
For each $1\le k\le N'+1$, we fix $\beta_k\in \{\alpha_k,\alpha_k^*\}$ such that $\la\beta_k,\nu\ra\geq 0$.
Then it follows from  $\sum_{\substack{\alpha \in \Pi}} \lambda_\alpha = -1$ that
$\sum_{\substack{\beta_k \neq \beta_k^*}} (2\ld_{\beta_k}+\la\beta_k,\nu\ra)+\sum_{\beta_k=\beta_k^*}\ld_{\beta_k}= -1$, i.e.,
\begin{align}\label{beta-sum}
2\sum_{\beta_k \neq \beta_k^*} (-\ld_{\beta_k})+\sum_{\beta_k=\beta_k^*}(-\ld_{\beta_k})
&= \Bigl(\sum_{\beta_k \neq \beta_k^*}\la\beta_k,\nu\ra\Bigr)+1.
\end{align}
If $\ld$ is small and $\la\beta_k,\nu\ra\geq 0$, then we have $(-\ld_{\beta_k})\geq 0$.
For fixed $\nu$, $\ld$ with $\ld^\natural=\nu$ is uniquely determined by $\ld_{\alpha_k}$ for $1\le k\le N'+1$ because $\ld(1)+\cdots+\ld(n)=0$.
Thus, it suffices to count the number of $\ld_{\beta_k}$ for $1\le k\le N'+1$ satisfying this equation.

For $\nu\in W_0\mu$, we define $\nu^\circ\in \Z^{N'}$ by $\nu^\circ(k)=\nu(i_k)=\nu(i_k^*)$ for $1\le k\le N'$.
Note that this gives a bijection between $\{\nu\mid \exists \ld\in \cAm,\ld^\natural=\nu\}$ and the set of all permutations of $(1^{(M)},0^{(N'-M)})$.
If $\alpha_1=\alpha_1^*$ (resp.\ $\alpha_{N+1}=\alpha_{N+1}^*$), let $\delta_1\coloneqq\nu^\circ(1)-1$ (resp.\ $\delta_{N+1} \coloneqq -\nu^\circ(N)$).
For $k \in \{1, N+1\}$, the relation $\ld(i_k^*)+\ld(i_k)=\mu(i_k)-\nu(i_k)$ implies that $\lambda_{\alpha_k}$ is even if $\delta_k=0$, and odd if $\delta_k=-1$. 
Combined with $\lambda_{\alpha_k}\le 0$, it follows that $\delta_k-\lambda_{\alpha_k}\ge 0$ can take only even values.
Set $l_k=-\ld_{\beta_k}$ if $\beta_k \neq \beta_k^*$, and $l_k=\delta_k-\ld_{\beta_k}$ if  $\alpha_k=\beta_k = \beta_k^*=\alpha_k^*$.
Then the above equation becomes
\begin{align}\tag{$1^\prime$}\label{lk-sum}
2\sum_{\beta_k \neq \beta_k^*} l_k+\sum_{\beta_k=\beta_k^*}l_k
&= \Bigl(\sum_{\beta_k \neq \beta_k^*}\la\beta_k,\nu\ra\Bigr)
 +\Bigl(\sum_{\beta_k= \beta_k^*}\delta_k\Bigr)+1,
\qquad l_k\geq 0.
\end{align}
Thus, the number of $\ld_{\beta_k}\le 0$ for $1\le k\le N'+1$ satisfying (\ref{beta-sum}) is equal to the number of $l_k\ge 0$ for $1\le k\le N'+1$ satisfying (\ref{lk-sum}).
Thus, $\sharp\cAm^{\sm}$ equals the number of choices of $\nu^\circ$ and nonnegative integers $l_1,\ldots,l_{N'+1}$ satisfying \eqref{lk-sum}, where $\nu^\circ$ runs through all permutations of $(1^{(M)},0^{(N'-M)})$ and $l_k$ is even if $\beta_k=\beta_k^*$.

We count $\nu^\circ$ and the $l_k$ simultaneously.
Since the case $N'=0$ is trivial, we may assume that $N'>0$.
For $\beta_k \neq \beta_k^*$, we have
$$
\langle\beta_k, \nu\rangle = 
\begin{cases}
1 - \nu^\circ(1) & (k = 1), \\
|\nu^\circ(k-1) - \nu^\circ(k)| & (2 \le k \le N'), \\
\nu^\circ(N') & (k = N'+1).
\end{cases}
$$
The number of changes from $0$ to $1$ minus the number of changes from $1$ to $0$ is $\nu^\circ(N')-\nu^\circ(1)$.
Let $\varepsilon_m,\varepsilon_{n-m}\in\{0,1\}$ be the residues of $m,n-m$ modulo $2$, respectively.
Set $l'_k=l_k$ if $\beta_k\neq\beta_k^*$, and $l'_k=l_k/2$ if $\beta_k=\beta_k^*$.
Then the $l'_k$ are nonnegative integers and \eqref{lk-sum} becomes
\begin{equation}\tag{$1^{\prime\prime}$}
\sum_{k=1}^{N'+1}l'_k=\varepsilon_m(1-\nu^\circ(1))+\varepsilon_{n-m}\nu^\circ(N')+\sharp\{1\le j<N'\mid \nu^\circ(j)>\nu^\circ(j+1)\}.
\end{equation}
Set $c(1)=1+l'_1-\varepsilon_m(1-\nu^\circ(1))\geq 1-\varepsilon_m$.
For $1\le i<N'$, define $c(i+1)$ by
$$c(i+1)-c(i)=\begin{cases}
l'_{i+1}&\text{if }\nu^\circ(i)>\nu^\circ(i+1),\\
1+l'_{i+1}&\text{otherwise}.
\end{cases}$$
Then $c(i+1)\geq c(i)$, with equality only if $\nu^\circ(i)=1$ and $\nu^\circ(i+1)=0$.
Moreover, we have $c(N')=N'+\varepsilon_{n-m}\nu^\circ(N')-l'_{N'+1}\leq N'+\varepsilon_{n-m}$.
Thus
$$\begin{aligned}
C_0&\coloneqq\{c(i)\mid \nu^\circ(i)=0\}\subseteq\{1-\varepsilon_m,\ldots,N'\},\\
C_1&\coloneqq\{c(i)\mid \nu^\circ(i)=1\}\subseteq\{1,\ldots,N'+\varepsilon_{n-m}\}
\end{aligned}$$
have cardinalities $N'-M$ and $M$, respectively.

Conversely, let $C_0\subseteq\{1-\varepsilon_m,\ldots,N'\}$ and $C_1\subseteq\{1,\ldots,N'+\varepsilon_{n-m}\}$ have cardinalities $N'-M$ and $M$, respectively.
Let $c(1)\le\cdots\le c(N')$ be the elements of $C_0$ and $C_1$, with each element of $C_0\cap C_1$ counted twice.
Set $\nu^\circ(i)=0$ for $c(i)\in C_0\setminus C_1$ and $\nu^\circ(i)=1$ for $c(i)\in C_1\setminus C_0$.
If $c(i)=c(i+1)\in C_0\cap C_1$, set $\nu^\circ(i)=1$ and $\nu^\circ(i+1)=0$.
We define $l'_1,\ldots,l'_{N'+1}$ from $c(1),\ldots,c(N')$ by the above equalities.
This gives a bijection between the choices of a permutation $\nu^\circ$ of $(1^{(M)},0^{(N'-M)})$ together with nonnegative integers $l'_1,\ldots,l'_{N'+1}$ satisfying $(1^{\prime\prime})$ and the pairs $(C_0,C_1)$.
Consequently,
$$\sharp\cAm^{\sm}=\binom{N'+\varepsilon_m}{N'-M}\binom{N'+\varepsilon_{n-m}}{M}
=\binom{\lfloor\frac{n}{2}\rfloor}{\lfloor\frac{m}{2}\rfloor}\binom{\lceil\frac{n}{2}\rceil}{\lceil\frac{m}{2}\rceil}.$$

\subsection{Type $D_n,{}^2D_n$ and $\mu=\omega_{n-1}^\vee,\omega_n^\vee$}
\label{type D}
We assume that $G$ is an absolutely simple group of adjoint type and $b\in \Omega\cap \vp^\mu W_0$.
Let $n\geq 5$.
The (absolute) root systems of type $D_n$ or ${}^2D_n$ is embedded in $V\coloneqq\bigoplus_{i=1}^n \mathbb{R} e_i$.
We identify $V$ with its dual via the standard Euclidean inner product. Under this identification, the minus simple roots and coroots coincide, and are given by $\alpha_i = e_{i+1}-e_i$ for $1 \leq i \leq n-1$ and $\alpha_n = -(e_{n-1} + e_n)$. 
The highest positive root is
$$
\alpha_0\coloneqq e_1+e_2=-\alpha_1-2\alpha_2-\cdots-2\alpha_{n-2}-\alpha_{n-1}-\alpha_n.
$$
The fundamental coweights are:
$$\omega_i^\vee = e_1 + \dots + e_i \quad (1 \leq i \leq n-2),$$$$\omega_{n-1}^\vee = \frac{1}{2}(e_1 + \dots + e_{n-1} - e_n),\quad\omega_n^\vee = \frac{1}{2}(e_1 + \dots + e_n).$$
Note that $W_0$ acts on $V$ by permuting the coordinates and changing an even number of signs.
Consequently, the orbit $W_0\mu$ consists of all vectors $(\pm \frac{1}{2}, \dots, \pm \frac{1}{2}) \in V$ having an odd (resp. even) number of negative coordinates if $\mu = \omega_{n-1}^\vee$ (resp. $\mu = \omega_n^\vee$).
The coweight lattice $\Y$ is given as $\Y = \bigoplus_{i=1}^n \mathbb{Z} e_i + \mathbb{Z}\left(\frac{1}{2}\sum_{i=1}^n e_i\right).$

Let $\tau_1, \tau_{n-1}, \tau_n \in \Omega$ be the length-zero elements in the Iwahori-Weyl group corresponding to the minuscule fundamental coweights $\omega_1^\vee, \omega_{n-1}^\vee$ and $\omega_n^\vee$, respectively.
These elements satisfy the relations $\tau_1 = \tau_n^2=\tau_{n-1}^2$ when $n$ is odd, and $\tau_1 = \tau_{n-1}\tau_n = \tau_n\tau_{n-1}$ when $n$ is even.
Thus, we have $\Omega \cong \mathbb{Z}/4\mathbb{Z}$ when $n$ is odd, and $\Omega \cong \mathbb{Z}/2\mathbb{Z} \times \mathbb{Z}/2\mathbb{Z}$ when $n$ is even.
The finite parts $p(\tau_i)$ act on $e_j$ as follows:
$$
\begin{gathered}
p(\tau_1)(e_j)=
\begin{cases}
-e_j & (j=1,n),\\
e_j & (1< j<n),
\end{cases}
\\[1ex]
\begin{aligned}
p(\tau_{n-1})(e_j)&=
\begin{cases}
e_n & (j=1),\\
-e_{n+1-j} & (1< j<n),\\
(-1)^n e_1 & (j=n),
\end{cases}
\quad
p(\tau_n)(e_j)=
\begin{cases}
-e_{n+1-j} & (1\le j<n),\\
(-1)^{n-1}e_1 & (j=n).
\end{cases}
\end{aligned}
\end{gathered}
$$
In type ${}^2D_n$, we have
\begin{align*}
\sigma(e_i)&=e_i \quad (1\le i\le n-1), & \sigma(e_n)&=-e_n,\\
\sigma(\alpha_i)&=\alpha_i \quad (1\le i\le n-2), &
\sigma(\alpha_{n-1})&=\alpha_n,\quad \sigma(\alpha_n)=\alpha_{n-1}.
\end{align*}
In particular, we have $\sigma(\tau_{n-1})=\tau_n$, $\sigma(\tau_n)=\tau_{n-1}$ and $\Omega\cap \J_b=\Omega^\sigma=\{1,\tau_1\}$.

Let $\mu=\omega_{n-1}^\vee$ or $\omega_n^\vee$ and $N=\lfloor \frac{n}{2}\rfloor$.
The $p(\tau_n\sigma)$-orbits on $\Pi$ are given as follows:
\begin{align*}
D_{2N}:&\quad
\{\alpha_1,\alpha_{2N-1}\},\ \{\alpha_{2N},\alpha_0\},\ \{\alpha_N\},\ 
\{\alpha_k,\alpha_{2N-k}\}\ (2\le k\le N-1),\\
D_{2N+1}:&\quad
\{\alpha_1,\alpha_{2N},\alpha_{2N+1},\alpha_0\},\ 
\{\alpha_k,\alpha_{2N+1-k}\}\ (2\le k\le N),\\
{}^2D_{2N}:&\quad
\{\alpha_1,\alpha_{2N-1},\alpha_{2N},\alpha_0\},\ \{\alpha_N\},\ 
\{\alpha_k,\alpha_{2N-k}\}\ (2\le k\le N-1),\\
{}^2D_{2N+1}:&\quad
\{\alpha_1,\alpha_{2N}\},\ \{\alpha_{2N+1},\alpha_0\},\ 
\{\alpha_k,\alpha_{2N+1-k}\}\ (2\le k\le N).
\end{align*}
The $p(\tau_{n-1}\sigma)$-orbits can be obtained by interchanging $\alpha_{n-1}$ and $\alpha_n$ in the above list.

For type ${}^2D_n$, we have $\nu^\diamond = (\nu(1), \dots, \nu(n-1), 0)$.
The condition $\nu^\diamond \le \nu_b=0$ requires $\sum_{i=1}^k \nu(i) \le 0$ for all $1 \le k \le n-1$. 
Since the final entry $\nu(n)$ can always be chosen to satisfy the parity condition of $W_0\mu$, it suffices to count the initial sequences of length $n-1$.
The number of such sequences with exactly $r\le \lfloor \frac{n-1}{2} \rfloor$ entries equal to $\frac{1}{2}$ among the $n-1$ indices is given by $\binom{n-1}{r} - \binom{n-1}{r-1}$.
Indeed, if a sequence fails the condition, then there is a first index $k$ where the number of $\frac{1}{2}$'s exceeds the number of $-\frac{1}{2}$'s. 
Exchanging $\frac{1}{2}$ and $-\frac{1}{2}$ in the initial segment up to that index $k$ gives a bijection with sequences having $r-1$ entries equal to $\frac{1}{2}$ and $n-r$ entries equal to $-\frac{1}{2}$.
Summing these terms over all $0 \le r \le \lfloor \frac{n-1}{2} \rfloor$ yields $\binom{n-1}{\lfloor \frac{n-1}{2} \rfloor}$.
Thus
$$\#\{\nu \in W_0\mu \mid \nu^\diamond \le \nu_b\} = \binom{n-1}{\lfloor \frac{n-1}{2} \rfloor}\quad\text{for type }{}^2D_n.$$

For type $D_n$, we have $\nu^\diamond = \nu$.
The condition $\nu \le \nu_b=0$ requires $\sum_{i=1}^k \nu(i) \le 0$ for $1 \le k \le n-1$ and $\sum_{i=1}^{n-1} \nu(i)\le -|\nu(n)|=-\frac{1}{2}$. 
The parity condition uniquely determines $\nu(n)$, so again we only count the initial sequences of length $n-1$.
If $n$ is even (resp.\ odd), the number of $\nu$ satisfying $\sum_{i=1}^k \nu(i) \le 0$ for $1 \le k \le n-1$ and $\sum_{i=1}^{n-1} \nu(i)=0$ is $0$ (resp.\ $\binom{n-1}{\frac{n-1}{2}}-\binom{n-1}{\frac{n-3}{2}}$).
Thus, by the ${}^2D_n$ case, we have
$$\#\{\nu \in W_0\mu \mid \nu^\diamond \le \nu_b\} = \binom{n-1}{\lfloor \frac{n-2}{2} \rfloor}\quad\text{for type }D_n.$$

We next count the $(\Omega\cap\J_b)$-orbits in $\cAm^{\sm}$.
Let $\nu=\ld^\natural$ for some $\ld\in \cAm$.
We first consider the $D_{2N}$ and ${}^2D_{2N+1}$ cases.
For each $0\le k\le N$, we fix $\beta_k\in\cO_{\alpha_k}$ such that $\la \beta_k,\nu\ra\geq 0$.
It follows from $\lambda_{\alpha_0} + \lambda_{\alpha_1} + 2\sum_{k=2}^{n-2} \lambda_{\alpha_k} + \lambda_{\alpha_{n-1}} + \lambda_{\alpha_n} = -1$ and Lemma \ref{ld} (1)
that
$$
2(-\lambda_{\beta_0})+2(-\lambda_{\beta_1})
+4\sum_{k=2}^{N-1}(-\lambda_{\beta_k})
+2(-\lambda_{\alpha_N})
=
\la\beta_0,\nu\ra+\la\beta_1,\nu\ra
+2\sum_{k=2}^{N-1}\la\beta_k,\nu\ra+1
$$
in the $D_{2N}$ case, and
$$
2(-\lambda_{\beta_0})+2(-\lambda_{\beta_1})
+4\sum_{k=2}^{N}(-\lambda_{\beta_k})
=
\la\beta_0,\nu\ra+\la\beta_1,\nu\ra
+2\sum_{k=2}^{N}\la\beta_k,\nu\ra+1
$$
in the ${}^2D_{2N+1}$ case.
Note that we must have $\la\beta_0,\nu\ra+\la\beta_1,\nu\ra=1$.
Note also that $\la\alpha_{n-k},\nu\ra=-\la\alpha_k,\nu\ra$ for $2\le k\le N$.
Write $l_k=-\ld_{\beta_k}$ for $0\le k\le N$.
Then we have
\begin{align*}
l_0+l_1+l_N+\sum_{k=2}^{N-1}\bigl(2l_k+1-\la\beta_k,\nu\ra\bigr)&=N-1,\\
l_0+l_1+\sum_{k=2}^{N}\bigl(2l_k+1-\la\beta_k,\nu\ra\bigr)&=N
\end{align*}
in the $D_{2N}$ and ${}^2D_{2N+1}$ cases, respectively.
For fixed $\nu$, the number of small $\ld$ with $\ld^\natural=\nu$ is equal to the number of $l_k\ge 0$ for $0\le k\le N$ satisfying these equations.

We claim that the map $\nu\mapsto(\la\beta_k,\nu\ra)_{k=2}^{N-1}$ (resp.\ $\nu\mapsto(\la\beta_k,\nu\ra)_{k=2}^{N}$) induces a bijection between $p(\Omega\cap\J_b)\backslash\{\nu\mid\exists\ld\in\cAm,\ \ld^\natural=\nu\}$ and $\{0,1\}^{N-2}$ (resp.\ $\{0,1\}^{N-1}$) in the $D_{2N}$ (resp.\ ${}^2D_{2N+1}$) case.
First note that $\nu=\ld^\natural$ implies $\nu(n)=2\mu(n)\nu(1)$ and $\nu(i)=\nu(n+1-i)$ for $2\le i\le n-1$.
Since $p(\tau_1)\nu=(-\nu(1),\nu(2),\ldots,\nu(n-1),-\nu(n))$, the $p(\Omega\cap\J_b)$-orbit of $\nu$ is determined by $\nu(2),\ldots,\nu(N+1)$.
The values $\la\beta_k,\nu\ra=|\la\alpha_k,\nu\ra|\in\{0,1\}$ for $2\le k\le N$ determine whether two adjacent entries of $\nu$ are equal or not.
Thus the vector $(\nu(2),\ldots,\nu(N+1))$ is determined up to sign.
We have $p(b)\nu=-\nu\in W_0\mu$ in the $D_{2N}$ case, and $\nu(N+1)=\frac12$ in the ${}^2D_{2N+1}$ case.
This proves the injectivity.
Conversely, given $(\la\beta_k,\nu\ra)_{k=2}^N$, we set $\nu(1)=\nu(N+1)=\frac12$, $\ld=\sum_{i=1}^N(\frac12-\nu(i))e_i$ and $\nu=\ld^\natural$.
In the $D_{2N}$ case, we use the convention $\la\beta_N,\nu\ra=0$.
It is straightforward to check that this construction proves the surjectivity.

The value of $2l_k+1-\la\beta_k,\nu\ra\ge0$ uniquely determines $l_k\ge0$ and $\la\beta_k,\nu\ra\in\{0,1\}$.
By the above bijection, a choice of an element of $(\Omega\cap\J_b)\backslash\cAm^{\sm}$ is equivalent to a choice of nonnegative integers $l_0,l_1,l_N$ and $\bigl(2l_k+1-\la\beta_k,\nu\ra\bigr)_{k=2}^{N-1}$ with total sum $N-1$ in the $D_{2N}$ case, or $l_0,l_1$ and $\bigl(2l_k+1-\la\beta_k,\nu\ra\bigr)_{k=2}^{N}$ with total sum $N$ in the ${}^2D_{2N+1}$ case.
Consequently,
$$\#\bigl((\Omega\cap\J_b)\backslash\cAm^{\sm}\bigr)=\binom{2N-1}{N-1}\ \left(\text{resp.\ }\binom{2N}{N}\right)\quad\text{for type }D_{2N}\ (\text{resp.\ }{}^2D_{2N+1}).$$

We next consider the $D_{2N+1}$ and ${}^2D_{2N}$ cases.
For each $2\le k\le N$, we fix $\beta_k\in\{\alpha_k,\alpha_{n-k}\}$ such that $\la \beta_k,\nu\ra\geq 0$.
It follows from Lemma \ref{ld} (1) that $\la\alpha_1+\alpha_{n-1}+\alpha_n+\alpha_0,\nu\ra=\sum_{i=0}^3\la p(b\sigma)^{-i}(\alpha_1),\nu\ra=0$.
By $\la\alpha_1+\alpha_0,\nu\ra=2\nu(2)$ and $\la\alpha_{n-1}+\alpha_n,\nu\ra=-2\nu(n-1)$, we deduce that $\nu(2)=\nu(n-1)$.
Since $\nu(2),\nu(n-1)\in \{\pm\frac{1}{2}\}$, the numbers $\langle\alpha_1,\nu\rangle$, $\langle\alpha_{n-1},\nu\rangle$, $\langle\alpha_n,\nu\rangle$ and $\langle\alpha_0,\nu\rangle$ are $1$, $-1$, $0$ and $0$ in some order.
Let $\beta_1\in\{\alpha_1,\alpha_{n-1},\alpha_n,\alpha_0\}$ such that $\la \beta_1,\nu\ra=1$.
It follows from $\lambda_{\alpha_0} + \lambda_{\alpha_1} + 2\sum_{k=2}^{n-2} \lambda_{\alpha_k} + \lambda_{\alpha_{n-1}} + \lambda_{\alpha_n} = -1$ and Lemma \ref{ld} (1)
that
$$ 4\sum_{k=1}^{N} (-\lambda_{\beta_k})= 3\langle\beta_1, \nu\rangle + 2\langle p(b\sigma)^{-1}\beta_1, \nu\rangle + \langle p(b\sigma)^{-2}\beta_1, \nu\rangle + 2\sum_{k=2}^{N} \langle\beta_k, \nu\rangle + 1 $$
in the $D_{2N+1}$ case, and
$$ 4\sum_{k=1}^{N-1} (-\lambda_{\beta_k}) + 2(-\lambda_{\alpha_N}) = 3\langle\beta_1, \nu\rangle + 2\langle p(b\sigma)^{-1}\beta_1, \nu\rangle + \langle p(b\sigma)^{-2}\beta_1, \nu\rangle + 2\sum_{k=2}^{N-1} \langle\beta_k, \nu\rangle + 1 $$
in the ${}^2D_{2N}$ case.
By $\la \beta_1,\nu\ra=1$, we must have $\langle p(b\sigma)^{-2}\beta_1, \nu\rangle=0$ and either $\langle p(b\sigma)^{-1}\beta_1, \nu\rangle=0$ or $-1$.
Write $l_k=-\lambda_{\beta_k}$ for $1\le k\le N$.
Then we have
\begin{align*}
2l_1-\la p(b\sigma)^{-1}\beta_1,\nu\ra
+\sum_{k=2}^{N}\bigl(2l_k+1-\la\beta_k,\nu\ra\bigr)&=N+1,\\
2l_1-\la p(b\sigma)^{-1}\beta_1,\nu\ra
+\sum_{k=2}^{N-1}\bigl(2l_k+1-\la\beta_k,\nu\ra\bigr)+l_N&=N
\end{align*}
in the $D_{2N+1}$ and ${}^2D_{2N}$ cases, respectively.
For fixed $\nu$, the number of small $\ld$ with $\ld^\natural=\nu$ is equal to the number of $l_k\ge 0$ for $1\le k\le N$ satisfying these equations.

{\emergencystretch=2em
We claim that the values $1+\la p(b\sigma)^{-1}\beta_1,\nu\ra$ and $(\la\beta_k,\nu\ra)_{k=2}^N$ (resp.\ $(\la\beta_k,\nu\ra)_{k=2}^{N-1}$) give a bijection between $p(\Omega\cap\J_b)\backslash\{\nu\mid\exists\ld\in\cAm,\ \ld^\natural=\nu\}$ and $\{(\varepsilon_k)_k\in\{0,1\}^N\mid\sum_{k=1}^N\varepsilon_k\equiv1\pmod2\}$ (resp.\ $\{0,1\}^{N-1}$) in the $D_{2N+1}$ (resp.\ ${}^2D_{2N}$) case.
First note that $\nu=\ld^\natural$ implies $\nu(i)=\nu(n+1-i)$ for $2\le i\le n-1$.
Since $p(\tau_1)\nu=(-\nu(1),\nu(2),\ldots,\nu(n-1),-\nu(n))$, the $p(\Omega\cap\J_b)$-orbit of $\nu$ is determined by $\nu(2),\ldots,\nu(N+1)$.
The values $\la\beta_k,\nu\ra=|\la\alpha_k,\nu\ra|\in\{0,1\}$ for $2\le k\le N$ determine whether two adjacent entries of $\nu$ are equal or not.
Thus the vector $(\nu(2),\ldots,\nu(N+1))$ is determined up to sign.
We have $(p(b)\nu)(i)=-\nu(i)$ for $2\le i\le n-1$ in the $D_{2N+1}$ case, and $1+\la p(b\sigma)^{-1}\beta_1,\nu\ra=\nu(2)+\frac12$ in the ${}^2D_{2N}$ case.
This proves the injectivity.
Conversely, given $1+\la p(b\sigma)^{-1}\beta_1,\nu\ra$ and $(\la\beta_k,\nu\ra)_{k=2}^N$, we set $\nu(1)=\frac12$, $\nu(2)=\frac12+\la p(b\sigma)^{-1}\beta_1,\nu\ra$, $\ld=\sum_{i=1}^N(\frac12-\nu(i))e_i$ and $\nu=\ld^\natural$.
In the ${}^2D_{2N}$ case, we use the convention $\la\beta_N,\nu\ra=0$.
It is straightforward to check that this construction proves the surjectivity.\par}

The value of $2l_1-\la p(b\sigma)^{-1}\beta_1,\nu\ra\ge0$ uniquely determines $l_1\ge0$ and $1+\la p(b\sigma)^{-1}\beta_1,\nu\ra\in\{0,1\}$, and the value of $2l_k+1-\la\beta_k,\nu\ra\ge0$ uniquely determines $l_k\ge0$ and $\la\beta_k,\nu\ra\in\{0,1\}$ as above.
By the above bijection, a choice of an element of $(\Omega\cap\J_b)\backslash\cAm^{\sm}$ is equivalent to a choice of nonnegative integers $2l_1-\la p(b\sigma)^{-1}\beta_1,\nu\ra$ and $\bigl(2l_k+1-\la\beta_k,\nu\ra\bigr)_{k=2}^{N}$ with total sum $N+1$ in the $D_{2N+1}$ case, or $l_N$, $2l_1-\la p(b\sigma)^{-1}\beta_1,\nu\ra$ and $\bigl(2l_k+1-\la\beta_k,\nu\ra\bigr)_{k=2}^{N-1}$ with total sum $N$ in the ${}^2D_{2N}$ case.
Consequently,
$$\#\bigl((\Omega\cap\J_b)\backslash\cAm^{\sm}\bigr)=\binom{2N}{N-1}\ \left(\text{resp.\ }\binom{2N-1}{N-1}\right)\quad\text{for type }D_{2N+1}\ (\text{resp.\ }{}^2D_{2N}).$$

\subsection{The Case of Coxeter Type and the Exceptional Cases}
Although the bijectivity of $\flat$ and the formula in Table \ref{count-small} can be verified easily by a direct computation in the case of Coxeter type (cf.\ Table \ref{coxeter-type} in \S\ref{weakly HN section} or \cite[Theorem 5.1.2]{GH15}), we may also argue as follows.
In this case, the Bruhat-Tits stratification coincides with the $\J$-stratification by \cite[Theorem 3.4]{Gortz19}.
By \cite[Corollary 4.6.2 and \S6]{GH15}, there exists a unique $\J_b$-orbit of Bruhat-Tits strata of a given dimension $d$, except for type ${}^2D_n$ with $d=n-1=\dim X_\mu(b)$; this exception is covered by the Chen-Zhu conjecture.
It is straightforward to check that $\sharp\{\nu\in W_0\mu\mid \nu^\diamond\le \nu_b\}$ is given by the formula in Table \ref{count-small} (see also \cite[\S7]{ZZ20}).
By Proposition \ref{dim} and Theorem \ref{main theo}, this implies that the map $\flat$ is bijective and yields the formula in Table \ref{count-small}.

The remaining cases are $D_4,{}^2D_4, {}^3D_4, E_6,{}^2E_6$ and $E_7$.
In these cases, one can verify Theorem \ref{flat bijective} by computer.
The root system of type $D_4$ is the same as in \S\ref{type D}. 
In type $D_4,{}^2D_4$ and ${}^3D_4$, the distinction between $3$ and $5$ in the third column of Table \ref{count-small} can be explained as follows.
We denote by $\sigma_{13},\sigma_{14},\sigma_{34}$ the involutions exchanging $\alpha_1$ and $\alpha_3$, $\alpha_1$ and $\alpha_4$, and $\alpha_3$ and $\alpha_4$, respectively.
We also denote by $\sigma_{134}$ and $\sigma_{143}$ the automorphisms of order $3$ given by
$$
\sigma_{134}\colon \alpha_1\mapsto\alpha_3\mapsto\alpha_4\mapsto\alpha_1,
\qquad
\sigma_{143}\colon \alpha_1\mapsto\alpha_4\mapsto\alpha_3\mapsto\alpha_1,
$$
and fixing $\alpha_2$.
There are $18$ pairs $(\mu,\sigma)$ in total. Among them, $9$ cases have $3$ small $\lambda$ up to $\Omega\cap \J_b$, and the other $9$ cases have $5$. Fix $\mu=\omega_i^\vee$ with $i\in\{1,3,4\}$. Then there are $3$ small $\lambda$ up to $\Omega\cap \J_b$ for $\sigma=\mathrm{id}$ and for the two involutions $\sigma$ whose subscripts contain $i$. There are $5$ small $\lambda$ up to $\Omega\cap \J_b$ for the remaining involution and for the two automorphisms $\sigma_{134},\sigma_{143}$ of order $3$.

\section{Weakly Fully Hodge-Newton Decomposability}
\label{weakly HN section}
We keep the assumption in \S\ref{numerical}.
See \cite[\S2]{SV25} for the notion of $\depth(G,\mu)\in \Q_{\ge 0}$.
The results of this section are summarized as follows.

\begin{theo}
\label{weakly HN}
Assume that $\mu$ is not central.
The following assertions are equivalent:
\begin{enumerate}[(1)]
\item  $\cS_{\mu, b}^{\pa}=\cS_{\mu, b}$;
\item  $\depth(G,\mu)<2$ or $b$ is superbasic;
\item $(G,\mu)$ is weakly fully Hodge-Newton decomposable.
\end{enumerate}
\end{theo}
In our setting, $\depth(G,\mu)\le 1$ if and only if $(G,\mu,K)$ is of Coxeter type (cf.\ Table \ref{coxeter-type}).
The classification of (2) and (3) is already known in \cite[Proposition 2.14]{CT22} and \cite[Theorem 2.7]{SV25} (cf.\ Table \ref{weakly HN classification}), which proves their equivalence.
Except for the type $D$ cases in Table \ref{weakly HN classification}, the implication $(2)\Rightarrow (1)$ follows from \cite[Proposition 3.4]{CV18}, \cite[Theorem 3.4]{Gortz19}, \cite[Theorem 3.4]{Shimada5} and \cite[Theorem 3.4]{ST24}.
The type $D$ cases in Table \ref{weakly HN classification} can be easily handled using Lemma \ref{every parahoric}.
Indeed, the Coxeter group $W_a\cap \J_b$ is of type $\tilde A_1$ in these cases.
Although one can also verify other cases using Lemma \ref{every parahoric}, we omit the details as it is straightforward.

\begin{table}[htbp]
\centering
\begin{minipage}{0.48\textwidth}
\centering
\renewcommand{\arraystretch}{1.05}
\[
\begin{array}{|c|c|}
\hline
\text{Type of }G & \text{Minuscule coweight }\mu \\
\hline
A_n, {}^2A_n
& \omega_1^\vee,\ \omega_n^\vee\\
\hline
B_n
& \omega_1^\vee \\
\hline
D_n, {}^2D_n
& \omega_1^\vee\\
\hline
D_4, {}^2D_4
& \omega_m^\vee \text{ s.t.\ }\omega_m^\vee= \sigma(\omega_m^\vee)  \\
\hline
A_3,{}^2A_3
& \omega_2^\vee \\
\hline
C_2
& \omega_2^\vee \\
\hline
\end{array}
\]
\caption{The case of Coxeter type}
\label{coxeter-type}
\end{minipage}\hfill
\begin{minipage}{0.48\textwidth}
\centering
\renewcommand{\arraystretch}{1.05}
\setlength{\tabcolsep}{5pt}
\[
\begin{array}{|c|c|}
\hline
\text{Type of }G & \text{Minuscule coweight }\mu \\
\hline
A_n(n\geq 4)
& \omega_2^\vee, \omega_{n-1}^\vee\\
\hline
A_n(n=5,6,7)
& \omega_3^\vee, \omega_{n-2}^\vee \\
\hline
C_3
& \omega_3^\vee \\
\hline
D_5
& \omega_5^\vee \\
\hline
{}^2D_4
& \omega_m^\vee \text{ s.t.\ }\omega_m^\vee\neq \sigma(\omega_m^\vee)\\
\hline
\end{array}
\]
\caption{The $1<\depth(G,\mu)<2$ cases}
\label{weakly HN classification}
\end{minipage}
\end{table}

Outside the classification of (2) and (3), it remains to show that (1) fails.
In the sequel, we show this case by case for each Dynkin type.
We will use the following lemma without further mention.
\begin{lemm}
\label{intersection}
If $V$ and $W$ are two irreducible subvarieties of a smooth variety $X$ with $V\cap W\neq\emptyset$, then
$$\dim(V\cap W)\geq \dim(V)+\dim(W)-\dim(X).$$
\end{lemm}
\begin{proof}
See \cite[Chapitre V (B) \S6]{Serre65} for example.
\end{proof}
Our strategy is to find an explicit $\ld_1\in \cAm^{\sm}$ violating the assumption of Lemma \ref{every parahoric}.
Then typically, there exists another $\ld\in \cAm^{\sm}$ such that $\sharp R(\ld)=\sharp R(\ld_1)+1$.
If $\cS_{\mu, b}^{\pa}=\cS_{\mu, b}$, then every $\J$-stratum in $X_\mu^{\ld_1}(b)$ has an associated small cocharacter.
In this case, we can explicitly remove from $X_\mu^{\ld_1}(b)$ the $\J$-strata of codimension $\geq 1$.
The remaining subset must be a single $\J$-stratum.
Then we deduce a contradiction by applying Lemma \ref{intersection} with $X=\Xl$ and using the existence of $w\in W_a\cap \J_b$ such that $w\ast \ld_1\neq \ld_1$.
Note that $\Xl$ admits a smooth deperfection by Proposition \ref{Xl}.

\subsection{Type $A_{n-1}$}
Without loss of generality, we may assume that $G=\GL_n$, $\mu=(1^{(m)},0^{(n-m)})$ and $b=\tau^m$.
Set $d=\gcd(m,n)$.
If $\min\{m,n-m\}\nmid n$, then $\chi_{d,n}^\vee\in \cAm^{\sm}$ and $\chi_{d,n}^\vee\notin \{\inv_K(1,w)\mid w\in \tW\cap \J_b\}$.
By Corollary \ref{par necessity}, we may also assume that $\min\{m,n-m\}\mid n$ and $\min\{m,n-m\}\geq 3$.
Let $\alpha_i=\chi_{i+1,i}$ and $\alpha_0=\chi_{1,n}$.

If $\min\{m,n-m\}=3$ and $n\geq 9$, then we set
$$\ld_1=(1,0,1,0,0,1,0^{(n-6)})\in \cAm^{\sm}\text{\quad and\quad } \ld_2=(0,1,1,0,1,0,0^{(n-6)})\in \cAm^{\sm}.$$
Set $\ld=w_{\alpha_1}\ld_1=w_{\alpha_2}\ld_2=(0,1,1,0,0,1,0^{(n-6)})\in \cAm^{\sm}$.
We have $\sharp R(\ld_1)=\sharp R(\ld_2)=3$ and $\sharp R(\ld)=4$.
Inside $X_\mu^\ld(b)$, we have $\dim(w_{\alpha_1}X_\mu^{\ld_1}(b)\cap w_{\alpha_2}X_\mu^{\ld_2}(b))=2$.
Suppose that $\cS_{\mu, b}^{\pa}=\cS_{\mu, b}$.
Then, by Proposition \ref{unique S} and Corollary \ref{par necessity}, the small cocharacter associated with a $\J$-stratum contained in $X_\mu^{\ld_1}(b)$ must be one of $\ld_1,w_{\alpha_2}\ld_1,w_{\alpha_1}w_{\alpha_2}\ld_1$ or $w_{\alpha_0}w_{\alpha_1}w_{\alpha_2}\ld_1=(1^{(3)},0^{(n-3)})$.
Thus $X_\mu^{\ld_1}(b)\setminus (I\cap \J_b)w_{\alpha_2}X_\mu^{\ld_3}(b)$ is a standard parahoric $\J$-stratum, where $\ld_3=w_{\alpha_2}\ld_1=(1,1,0,0,1,0,0^{(n-6)})\in \cAm^{\sm}$.
Note that $\sharp R(\ld_3)=2$.
Again by Proposition \ref{unique S}, we have
$$w_{\alpha_1}X_\mu^{\ld_1}(b)\cap w_{\alpha_2}X_\mu^{\ld_2}(b)=w_{\alpha_1}(I\cap \J_b)w_{\alpha_2}X_\mu^{\ld_3}(b)\cap w_{\alpha_2}X_\mu^{\ld_2}(b).$$
Let $S$ be the unique standard parahoric $\J$-stratum in $X_\mu^{\ld_3}(b)$.
Then $$(I\cap \J_b)w_{\alpha_2}w_{\alpha_1}w_{\alpha_2}S=(I\cap \J_b)w_{\alpha_1}w_{\alpha_2}w_{\alpha_1}S=(I\cap \J_b)w_{\alpha_1}w_{\alpha_2}S\subset \Xl.$$
This implies that $\dim(w_{\alpha_1}(I\cap \J_b)w_{\alpha_2}X_\mu^{\ld_3}(b)\cap w_{\alpha_2}X_\mu^{\ld_2}(b))\le 1$, which is a contradiction.
Hence we must have $\cS_{\mu, b}^{\pa}\neq\cS_{\mu, b}$.

If $\min\{m,n-m\}\geq 4$, then we set
$$\ld_1=(1,0,1,0,1,0,0^{(n-6)})\in \cAm^{\sm}\quad\text{and}\quad \ld=w_{\alpha_3}\ld_1=(1,0,0,1,1,0,0^{(n-6)})\in \cAm^{\sm}.$$
We have $\sharp R(\ld_1)=3$ and $\sharp R(\ld)=4$.
Inside $X_\mu^\ld(b)$, we have $\dim(w_{\alpha_3}X_\mu^{\ld_1}(b)\cap w_{\alpha_2}w_{\alpha_4}w_{\alpha_3}X_\mu^{\ld_1}(b))=2$.
Suppose that $\cS_{\mu, b}^{\pa}=\cS_{\mu, b}$.
Then, by Proposition \ref{unique S} and Corollary \ref{par necessity}, the small cocharacter associated with a $\J$-stratum contained in $X_\mu^{\ld_1}(b)$ must be one of $\ld_1,w_{\alpha_2}\ld_1,w_{\alpha_4}\ld_1,w_{\alpha_2}w_{\alpha_4}\ld_1$ or $w_{\alpha_3}w_{\alpha_2}w_{\alpha_4}\ld_1=(1^{(3)},0^{(n-3)})$.
Thus $X_\mu^{\ld_1}(b)\setminus \bigl((I\cap \J_b)w_{\alpha_2}X_\mu^{\ld_2}(b)\cup  (I\cap \J_b)w_{\alpha_4}X_\mu^{\ld_3}(b)\bigr)$ is a standard parahoric $\J$-stratum, where $\ld_2=w_{\alpha_2}\ld_1=(1,1,0,0,1,0,0^{(n-6)})\in \cAm^{\sm}$ and $\ld_3=w_{\alpha_4}\ld_1=(1,0,1,1,0,0,0^{(n-6)})\in \cAm^{\sm}$.
Note that $\sharp R(\ld_2)=\sharp R(\ld_3)=2$.
Again by Proposition \ref{unique S}, we have
\begin{align*}
&w_{\alpha_3}X_\mu^{\ld_1}(b)\cap w_{\alpha_2}w_{\alpha_4}w_{\alpha_3}X_\mu^{\ld_1}(b)\\
=&\bigl(w_{\alpha_3}(I\cap \J_b)(w_{\alpha_2}X_\mu^{\ld_2}(b)\cup w_{\alpha_4}X_\mu^{\ld_3}(b))\bigr)\cap w_{\alpha_2}w_{\alpha_4}w_{\alpha_3}X_\mu^{\ld_1}(b).
\end{align*}
Let $S\subset X_\mu^{\ld_2}(b)$ and $S’\subset X_\mu^{\ld_3}(b)$ be standard parahoric $\J$-strata.
Then \begin{align*}
(I\cap \J_b)w_{\alpha_3}w_{\alpha_4}w_{\alpha_2}w_{\alpha_3}w_{\alpha_2}S\subset \Xl,\quad
(I\cap \J_b)w_{\alpha_3}w_{\alpha_4}w_{\alpha_2}w_{\alpha_3}w_{\alpha_4}S'\subset \Xl.
\end{align*}
This implies that $\dim\bigl(\bigl(w_{\alpha_3}(I\cap \J_b)(w_{\alpha_2}X_\mu^{\ld_2}(b)\cup w_{\alpha_4}X_\mu^{\ld_3}(b))\bigr)\cap w_{\alpha_2}w_{\alpha_4}w_{\alpha_3}X_\mu^{\ld_1}(b)\bigr)\le 1$, which is a contradiction.
Hence we must have $\cS_{\mu, b}^{\pa}\neq\cS_{\mu, b}$.

\subsection{Type $C_n$}
Let $n\geq 4$.
Without loss of generality, we may assume that $G=\GSp_{2n}$, $\mu=\omega_n^\vee=(1^{(n)},0^{(n)})$ and $b=\vp^\mu(1\ n+1)\cdots (n\ 2n)$ as in \S\ref{Cn}.
We set $$\ld_1=(0,1,0^{(2n-4)},-1,0)\in \cAm^{\sm}\text{\quad and\quad} \ld=w_{\alpha_0}\ld_1=(1,1,0^{(2n-4)},-1,-1)\in \cAm^{\sm},$$
where $w_{\alpha_0}=s_0s_n$.
We have $\sharp R(\ld_1)=2$ and $\sharp R(\ld)=3$.
Inside $\Xl$, we have $\dim(w_{\alpha_0}X_\mu^{\ld_1}(b)\cap w_{\alpha_1}w_{\alpha_0}X_\mu^{\ld_1}(b))=1$, where $w_{\alpha_1}=s_1s_{n-1}$.
Suppose that $\cS_{\mu, b}^{\pa}=\cS_{\mu, b}$.
Then, by Proposition \ref{unique S} and Corollary \ref{par necessity}, the small cocharacter associated with a $\J$-stratum contained in $X_\mu^{\ld_1}(b)$ must be one of $\ld_1,w_{\alpha_1}\ld_1$ or $w_{\alpha_0}w_{\alpha_1}\ld_1=0$.
Thus $X_\mu^{\ld_1}(b)\setminus (I\cap \J_b)w_{\alpha_1}X_\mu^{\ld_2}(b)$ is a standard parahoric $\J$-stratum, where $\ld_2=w_{\alpha_1}\ld_1=(1,0^{(2n-2)},-1)\in \cAm^{\sm}$.
Note that $\sharp R(\ld_2)=1$.
Again by Proposition \ref{unique S}, we have
$$w_{\alpha_0}X_\mu^{\ld_1}(b)\cap w_{\alpha_1}w_{\alpha_0}X_\mu^{\ld_1}(b)=w_{\alpha_0}(I\cap \J_b)w_{\alpha_1}X_\mu^{\ld_2}(b)\cap w_{\alpha_1}w_{\alpha_0}X_\mu^{\ld_1}(b).$$
Let $S$ be the unique standard parahoric $\J$-stratum in $X_\mu^{\ld_2}(b)$.
Then $$(I\cap \J_b)w_{\alpha_0}w_{\alpha_1}w_{\alpha_0}w_{\alpha_1}S=(I\cap \J_b)w_{\alpha_1}w_{\alpha_0}w_{\alpha_1}w_{\alpha_0}S=(I\cap \J_b)w_{\alpha_1}w_{\alpha_0}w_{\alpha_1}S\subset \Xl.$$
This implies that $\dim(w_{\alpha_0}(I\cap \J_b)w_{\alpha_1}X_\mu^{\ld_2}(b)\cap w_{\alpha_1}w_{\alpha_0}X_\mu^{\ld_1}(b))= 0$, which is a contradiction.
Hence we must have $\cS_{\mu, b}^{\pa}\neq\cS_{\mu, b}$.

\subsection{Type ${}^2A_{n-1}$}
We follow the description in \S\ref{2An}.

If $m=2$ and $n\geq 6$, then $w_{\alpha_0}=s_0s_2$, $w_{\alpha_1}=s_1$, $w_{\alpha_3}=s_3s_{n-1}$ and we set
\begin{align*}
\ld_1&=(1,0,0,1,0^{(n-6)},-1,-1)\in \cAm^{\sm}\\
\text{and}\quad\ld&=w_{\alpha_0}w_{\alpha_1}\ld_1=(0,0,1,1,0^{(n-6)},-1,-1)\in \cAm^{\sm}.
\end{align*}
We have $\sharp R(\ld_1)=3$ and $\sharp R(\ld)=4$.
Inside $X_\mu^\ld(b)$, we have $\dim(w_{\alpha_0}w_{\alpha_1}X_\mu^{\ld_1}(b)\cap w_{\alpha_3}w_{\alpha_1}w_{\alpha_0}w_{\alpha_1}X_\mu^{\ld_1}(b))=2$.
In this case, one can prove that $\cS_{\mu,b}^{\pa}\neq \cS_{\mu,b}$ as in the $A_{n-1}$ case with $\min\{m,n-m\}\geq 4$.

If $m=n-2\geq 4$, then $w_{\alpha_0}=s_0s_{n-2}$, $w_{\alpha_{n-3}}=s_1s_{n-3}$, $w_{\alpha_{n-1}}=s_{n-1}$ and we set
\begin{align*}
\ld_1&=(1,1,0^{(n-6)},-1,0,0,-1)\in \cAm^{\sm}\\
\text{and}\quad\ld&=w_{\alpha_0}w_{\alpha_{n-1}}\ld_1=(1,1,0^{(n-6)},-1,-1,0,0)\in \cAm^{\sm}.
\end{align*}
We have $\sharp R(\ld_1)=3$ and $\sharp R(\ld)=4$.
Inside $X_\mu^\ld(b)$, we have $\dim(w_{\alpha_0}w_{\alpha_{n-1}}X_\mu^{\ld_1}(b)\cap w_{\alpha_{n-3}}w_{\alpha_{n-1}}w_{\alpha_0}w_{\alpha_{n-1}}X_\mu^{\ld_1}(b))=2$.
In this case, one can prove that $\cS_{\mu,b}^{\pa}\neq \cS_{\mu,b}$ as in the $A_{n-1}$ case with $\min\{m,n-m\}\geq 4$.

If $m=3$ and $n\geq 5$, then $w_{\alpha_0}=s_0s_3$, $w_{\alpha_1}=s_1s_2s_1$ and we set
$$\ld_1=(0,0,1,0^{(n-4)},-1)\in \cAm^{\sm}\quad\text{and}\quad\ld=w_{\alpha_0}\ld_1=(0,0,0,1,0^{(n-5)},-1)\in \cAm^{\sm}.$$
We have $\sharp R(\ld_1)=2$ and $\sharp R(\ld)=3$.
Inside $X_\mu^\ld(b)$, we have $\dim(w_{\alpha_0}X_\mu^{\ld_1}(b)\cap w_{\alpha_1}w_{\alpha_0}X_\mu^{\ld_1}(b))=1$.
In this case, one can prove that $\cS_{\mu,b}^{\pa}\neq \cS_{\mu,b}$ as in the $C_n$ case.

If $m=n-3\geq 2$, then $w_{\alpha_0}=s_0s_{n-3}$, $w_{\alpha_{n-1}}=s_{n-2}s_{n-1}s_{n-2}$ and we set
$$\ld_1=(1,0^{(n-4)},-1,0,0)\in \cAm^{\sm}\quad\text{and}\quad\ld=w_{\alpha_0}\ld_1=(1,0^{(n-5)},-1,0,0,0)\in \cAm^{\sm}.$$
We have $\sharp R(\ld_1)=2$ and $\sharp R(\ld)=3$.
Inside $X_\mu^\ld(b)$, we have $\dim(w_{\alpha_0}X_\mu^{\ld_1}(b)\cap w_{\alpha_{n-1}}w_{\alpha_0}X_\mu^{\ld_1}(b))=1$.
In this case, one can prove that $\cS_{\mu,b}^{\pa}\neq \cS_{\mu,b}$ as in the $C_n$ case.

If $\min\{m,n-m\}\geq 4$, then we set
$$\ld_1=(0,1,0^{(n-4)},-1,0)\in \cAm^{\sm}\quad\text{and}\quad\ld=w_{\alpha_0}\ld_1=(1,1,0^{(n-4)},-1,-1)\in \cAm^{\sm}.$$
We have $\sharp R(\ld_1)=3$ and $\sharp R(\ld)=4$.
Inside $X_\mu^\ld(b)$, we have $\dim(w_{\alpha_0}X_\mu^{\ld_1}(b)\cap w_{\alpha_{n-1}}w_{\alpha_1}w_{\alpha_0}X_\mu^{\ld_1}(b))=2$.
In this case, one can prove that $\cS_{\mu,b}^{\pa}\neq \cS_{\mu,b}$ as in the $A_{n-1}$ case with $\min\{m,n-m\}\geq 4$.

\subsection{Type $D_n,{}^2D_n$ and $\mu=\omega_{n-1}^\vee,\omega_n^\vee$}
We assume that $G$ is an absolutely simple group of adjoint type, $\mu=\omega_{n-1}^\vee$ or $\omega_n^\vee$, and $b\in \Omega\cap \vp^\mu W_0$.
We follow the description of the root system of type $D_n$ or ${}^2D_n$ in \S\ref{type D}.
For $1\le i\le n-1$, $s_i$ swaps the $i$-th and $(i+1)$-th coordinates, while $s_n$ swaps the $(n-1)$-th and $n$-th coordinates and multiplies them by $-1$. The affine reflection $s_0$ swaps the first and second coordinates, multiplies them by $-1$, and then adds $e_1+e_2$.

If $n=5$ and $w_{\alpha_0}=s_0s_{n-1}$, we set
$$\ld_1=(0,0,0,1,0)\in\cAm^{\sm},\quad \ld=w_{\alpha_1}\ld_1\in \cAm^{\sm}\text{\quad and\quad} \ld_2=w_{\alpha_2}\ld_1\in \cAm^{\sm}.$$
We have $\sharp R(\ld_1)=2,\sharp R(\ld)=3,\sharp R(\ld_2)=1$ and $\dim(w_{\alpha_1}X_\mu^{\ld_1}(b)\cap w_{\alpha_2}w_{\alpha_1}X_\mu^{\ld_1}(b))=1$ inside $\Xl$.
In this case, one can prove that $\cS_{\mu,b}^{\pa}\neq \cS_{\mu,b}$ as in the $C_n$ case.

If $n\geq 6$, we set
$$\ld_1=(0,0,1,0^{(n-3)})\in\cAm^{\sm},\quad \ld=w_{\alpha_0}\ld_1\in \cAm^{\sm},\text{\quad and\quad} \ld_2=w_{\alpha_2}\ld_1\in \cAm^{\sm}.$$
Then $\sharp R(\ld_1)=2,\sharp R(\ld)=3$ and $\sharp R(\ld_2)=1$.
If $w_{\alpha_0}=s_0s_1s_{n-1}s_n$, we have $\dim(w_{\alpha_0}X_\mu^{\ld_1}(b)\cap w_{\alpha_2}w_{\alpha_0}X_\mu^{\ld_1}(b))=1$ inside $\Xl$.
In this case, one can prove that $\cS_{\mu,b}^{\pa}\neq \cS_{\mu,b}$ as in the $C_n$ case.

Assume that $n\geq 6$ and $w_{\alpha_0}=s_0s_{n-1}$.
We adopt another approach in this case.
Note that $$\dim\bigl(U_{\alpha_2}(\cO)\vp^{\ld_1}K/K\cap (I\cap\J_b)w_{\alpha_2}X_\mu^{\ld_2}(b)\bigr)=0.$$
In particular, the intersection of $U_{\alpha_2}(\cO)\vp^{\ld_1}K/K$ and $X_\mu^{\ld_1}(b)\setminus (I\cap \J_b)w_{\alpha_2}X_\mu^{\ld_2}(b)$ is non-empty.
On the other hand, we have $w_{\alpha_2}w_{\alpha_1}\alpha_2=\alpha_1$, $w_{\alpha_2}w_{\alpha_1}\ld_1=\ld_2$ and $(w_{\alpha_2}w_{\alpha_1})\ast\ld_1=\ld_1$.
This implies that
$$w_{\alpha_2}w_{\alpha_1}U_{\alpha_2}(\cO)\vp^{\ld_1}K/K\subset I\vp^{\ld_2}K/K\quad\text{and}\quad w_{\alpha_2}w_{\alpha_1}S\subset I\vp^{\ld_1}K/K,$$
where $\ld_2=(0,1,0^{(n-2)})\in\cAm^{\sm}$ and $S$ is the unique standard parahoric $\J$-stratum in $X_\mu^{\ld_1}(b)$.
Hence $X_\mu^{\ld_1}(b)\setminus (I\cap \J_b)w_{\alpha_2}X_\mu^{\ld_2}(b)$ is not a standard parahoric $\J$-stratum.
Thus $\cS_{\mu,b}^{\pa}\neq \cS_{\mu,b}$.

\subsection{The Exceptional Cases}
The remaining cases are ${}^3D_4, E_6,{}^2E_6$ and $E_7$.
In the ${}^3D_4$ (resp.\ $E_6$ or ${}^2E_6$, resp.\ $E_7$) case, let $\ld_1$ be the unique small cocharacter up to the action of $\Omega\cap \J_b$ such that $\sharp R(\ld_1)=2$ (resp.\ $3$, resp.\ $4$).
Using this $\ld_1$, we can argue as above in all of these cases.
We omit the details as it is straightforward.

\bibliographystyle{myamsplain}
\bibliography{reference}
\end{document}